\documentclass{amsart}

\usepackage{placeins}

\usepackage{hyphenat}
\usepackage{amssymb, amsmath, amsthm, amsfonts, mathtools}
\usepackage{mathabx}

\allowdisplaybreaks

\usepackage{hyperref}
\usepackage[nameinlink]{cleveref}

\usepackage{autonum} 

\usepackage{xcolor}

\usepackage{tensor}

\usepackage{wasysym}

\usepackage[normalem]{ulem}

\usepackage{pdfpages}

\usepackage{graphicx}

\usepackage{marginnote}

\usepackage{xfrac}
\usepackage{nicefrac}
\usepackage{faktor}

\usepackage{cancel}

\usepackage{slashed}

\usepackage{tikz}
\usetikzlibrary{matrix,calc,arrows.meta}
\usepackage{tikz-cd}

\usepackage{comment}

\usepackage{esvect}

\usepackage{braket}

\usepackage{enumitem}

\usepackage{nicematrix}

\usepackage{thmtools}

\usepackage{pdfpages}

\usepackage{leftindex}

\usepackage{upgreek}

\theoremstyle{plain}
\newtheorem{theorem}{Theorem}[section]
\crefname{theorem}{Theorem}{Theorems}
\newtheorem{proposition}[theorem]{Proposition}
\crefname{proposition}{Proposition}{Propositions}
\newtheorem{lemma}[theorem]{Lemma}
\crefname{lemma}{Lemma}{Lemmas}
\newtheorem{corollary}[theorem]{Corollary}
\crefname{corollary}{Corollary}{Corollaries}
\theoremstyle{definition}
\newtheorem{definition}[theorem]{Definition}
\crefname{definition}{Definition}{Definitions}
\newtheorem{example}[theorem]{Example}
\crefname{example}{Example}{Examples}
\newtheorem{def-prop}[theorem]{Definition-Proposition}
\crefname{def-prop}{Definition-Proposition}{}
\theoremstyle{remark}
\newtheorem{remark}[theorem]{Remark}
\crefname{remark}{Remark}{Remarks}

\crefname{claim}{Claim}{Claims}
\newtheorem{notation}[theorem]{Notation}
\crefname{notation}{Notation}{}

\crefname{question}{Question}{Questions}

\crefname{chapter}{Chapter}{Chapters}
\crefname{section}{Section}{Sections}
\crefname{subsection}{Subsection}{Subsections}
\crefname{figure}{Figure}{Figures}

\usepackage[figurename=Figure,tablename=Table]{caption}

\newcommand*\colvec[3][]{
	\begin{pmatrix}\ifx\relax#1\relax\else#1\\\fi#2\\#3\end{pmatrix}
}

\DeclareMathOperator{\Dom}{Dom}

\renewcommand{\Re}{\text{Re}}
\renewcommand{\Im}{\text{Im}}

\usepackage{mathbbol} 
\DeclareSymbolFontAlphabet{\mathbbm}{bbold}%
\DeclareSymbolFontAlphabet{\mathbb}{AMSb}%

\def\black{\color{black}}

\title[Peter--Weyl truncations of compact quantum groups]{Convergence of Peter--Weyl Truncations of Compact Quantum Groups}

\author{Malte Leimbach}
\address{Institute for Mathematics, Astrophysics and Particle Physics, Radboud
		University Nijmegen, Heyendaalseweg 135, 65254 AJ Nijmegen, The Netherlands.
	}
\email{m.leimbach@math.ru.nl}
\date{April 10, 2025}

\begin{document}
	
	\begin{abstract}
		We consider a coamenable compact quantum group $\mathbb{G}$ as a compact quantum metric space if its function algebra $\mathrm{C}(\mathbb{G})$ is equipped with a Lip-norm.
		By using a projection $P$ onto direct summands of the Peter--Weyl decomposition, the $\mathrm{C}^*$-algebra $\mathrm{C}(\mathbb{G})$ can be compressed to an operator system $P\mathrm{C}(\mathbb{G})P$, and there are induced left and right coactions on this operator system.
		Assuming that the Lip-norm on $\mathrm{C}(\mathbb{G})$ is bi-invariant in the sense of Li, there is an induced bi-invariant Lip-norm on the operator system $P\mathrm{C}(\mathbb{G})P$ turning it into a compact quantum metric space.
		Given an appropriate net of such projections which converges strongly to the identity map on the Hilbert space $\mathrm{L}^2(\mathbb{G})$, we obtain a net of compact quantum metric spaces.
		We prove convergence of such nets in terms of Kerr's complete Gromov--Hausdorff distance.
		An important tool is the choice of an appropriate state whose induced slice map gives an approximate inverse of the compression map $\mathrm{C}(\mathbb{G}) \ni a \mapsto PaP$ in Lip-norm.
	\end{abstract}
	
	\maketitle
	
	\tableofcontents

	\section{Introduction}
	
	A spectral approach to geometry allows for generalizations to the noncommutative realm \cite{Con94} in terms of spectral triples, which indeed recover a Riemannian spin-manifold in the commutative case \cite{Con96, Con13}.
	One of the inputs for reconstructing a Riemannian spin manifold $M$ from its associated spectral triple $(\mathrm{C}^\infty(M),\mathrm{L}^2(S_M),D_M)$ is the full spectrum of the spin-Dirac operator $D_M$.
	As argued in \cite{CvS21}, it is, however, physically more reasonable to expect only availability of part of the spectral data.
	The authors point out that in that case one is led to consider operator systems rather than $\mathrm{C}^*$-algebras and put forward the notion of \emph{spectral truncations}.
	More precisely, given a metric spectral triple $(A,H,D)$ and a family of spectral projections $P_\Lambda : H \rightarrow H$ associated to the operator $D$ and which converges strongly to the identity map $\mathbf{I}^H : H \rightarrow H$, one may consider the family of \emph{operator system spectral triples} $(P_\Lambda A P_\Lambda,P_\Lambda H,P_\Lambda D P_\Lambda)$ and ask about convergence to the spectral triple $(A,H,D)$.
	The by now established setting in which this issue can be reasonably addressed is that of Rieffel's compact quantum metric spaces \cite{Rie98, Rie99} and quantum versions of Gromov--Hausdorff distance \cite{Rie04, Ker03, KL07}.
	In this sense, convergence of spectral truncations has been proven for the circle \cite{vS21} and for tori \cite{LvS24}.
	See also the related work \cite{Toy23}.
	
	In this article we discuss convergence of truncations of coamenable compact quantum groups.
	As spectral data we consider, however, the irreducible finite dimensional corepresentations, rather than the spectrum of a Dirac operator.
	We therefore call the truncations under consideration \emph{Peter--Weyl truncations}.
	Our corepresentation theoretic setting seems to be easier to approach than one involving Dirac operators and it allows us to generalize techniques used for Peter--Weyl truncations of compact metric groups \cite{GEvS23}.
	
	Peter--Weyl truncations are complementary to Fourier truncations \cite{Rie23}.
	We point out that the operator systems arising in these two approaches are quite different.
	In fact, the Toeplitz system which arises as the Peter--Weyl truncation of the circle has propagation number $2$, whereas the Fejér--Riesz system obtained as the Fourier truncation of the circle has propagation number $\infty$, see \cite{CvS21} for the definition of the propagation number and proofs of these facts.
	In particular, these operator systems are not Morita equivalent in the sense of \cite{EKT22}.
	Also the matter of duality is still under investigation \cite[Subsection 4.2]{LvS24}.
	
	We give a brief sketch of our line of argument for convergence of Peter--Weyl truncations of compact quantum groups.
	Given a coamenable compact quantum group $\mathbb{G}$, we assume that its function algebra $A := \mathrm{C}(\mathbb{G})$ is equipped with a Lip-norm $L_A$ to give it the structure of a compact quantum metric space. 
	It is crucial that the Lip-norm is invariant for the left and right coactions by comultiplication of $A$ on itself.
	The notion of (right) invariance was put forward in \cite{Li09} and means that $L_A((\mu \otimes \mathbf{I}^A)\Delta(a)) \leq L_A(a)$, for all elements $a \in A$ and states $\mu \in \mathcal{S}(A)$.
	Let $P_\Lambda : \mathrm{L}^2(\mathbb{G}) \rightarrow \mathrm{L}^2(\mathbb{G})$ be a projection associated to the Peter--Weyl decomposition of the compact quantum group $\mathbb{G}$.
	Upon realizing that the comultiplication induces ergodic left and right coactions on the compression $A^{(\Lambda)} := P_\Lambda A P_\Lambda$ (\autoref{lem:Induced-Action-on-Toeplitz-System}), we apply one of the main results from \cite{Li09}, recalled in \autoref{prop:Li-Thm-1.4}, to obtain an induced bi-invariant Lip-norm on the operator system $A^{(\Lambda)}$.
	
	We emphasize that we consider convergence in Kerr's \emph{complete Gromov--Hausdorff distance} \cite{Ker03, KL07} and we check that the proof of a criterion for control of this distance \cite{KK22} extends to the complete setting (\autoref{prop:Criterion-qGH-comparison}).
	This method goes back to Rieffel's early papers, where it was formalized in terms of \emph{bridges} \cite{Rie04}, and to the idea from \cite{vS21} of finding appropriate morphisms $\tau : A \rightarrow A^{(\Lambda)}$, $\sigma : A^{(\Lambda)} \rightarrow A$ of compact quantum metric spaces.
	Complete Gromov--Hausdorff distance (rather than quantum Gromov--Hausdorff distance) seems to be suitable in the setting of operator systems (rather than order unit spaces) as it takes the matrix order structure into account.
	
	In view of the criterion just mentioned, we propose for the map $\tau : A \rightarrow A^{(\Lambda)}$ the compression map $a \mapsto P_\Lambda a P_\Lambda$ and, similarly as in \cite{GEvS23}, for the map $\sigma : A^{(\Lambda)} \rightarrow A$ the slice map $x \mapsto (\phi \otimes \mathbf{I}^A) \alpha^\tau(x)$, for an appropriate choice of a state $\phi \in \mathcal{S}(A^{(\Lambda)})$ and for $\alpha^\tau : A^{(\Lambda)} \rightarrow A^{(\Lambda)} \otimes A$ the above-mentioned coaction induced by the comultiplication.
	Invariance of the Lip-norms yields straightforwardly that these maps are morphisms of compact quantum metric spaces, see \autoref{lem:Compression-map-morphism}, \autoref{lem:Symbol-map-morphism}.
	Their compositions can be compared to the respective identity maps on $A$ and $A^{(\Lambda)}$ in terms of Lip-norms by a general argument about slice maps (\autoref{lem:Lip-estimate-for-action-OS}).
	A density result for states, \autoref{lem:Weak*-Density-Of-Liftable-States}, now is enough to satisfy the criterion in \autoref{prop:Criterion-qGH-comparison} and thus to prove our main theorem, \autoref{thm:Convergence-PW-truncations}, that bi-invariant Peter--Weyl truncations of coamenable compact quantum groups converge in complete Gromov--Hausdorff distance.
	
	The article is organized as follows.
	\autoref{sec:Operator-systems} and \autoref{sec:Compact-quantum-groups} are preliminary with the main purpose of fixing terminology and notation on operator systems and compact quantum groups. 
	We decided to include some details about coactions on operator systems in \autoref{sec:Coactions}, compact quantum metric spaces and complete Gromov--Hausdorff distance in \autoref{sec:CQMS}, as well as invariant Lip-norms in \autoref{sec:Invariant-Lip-norms}, as some of them have not yet been combined in the literature in the way necessary for our purposes.
	Our main arguments are in \autoref{sec:PW-truncations} and the experienced reader will be able to follow them by only referring back to the earlier sections as needed.
	We also explain in some detail how our result generalizes that of \cite{GEvS23}.

	\section*{Acknowledgements}
	
	I would like to thank Yvann Gaudillot-Estrada, Jens Kaad, Marc Rieffel and Walter van Suijlekom for many helpful discussions, the Mathematics department of UC Berkeley for their hospitality during my visit in spring 2024, and an anonymous referee for careful suggestions.
	This work was funded by NWO under grant OCENW.KLEIN.376.
	This article is based upon work from COST Action CaLISTA CA21109 supported by COST (European Cooperation in Science and Technology).

	\section{Operator systems}\label{sec:Operator-systems}
	
	We collect a few properties of operator systems will be used throughout without further reference.
	For more details, see e.g.\@ \cite{Pau02}.
	
	We only consider \emph{unital} operator systems and we usually work with their \emph{concrete} version, i.e.\@ for us an operator system is a unital $^*$-closed subspace of $\mathcal{B}(H)$, for some Hilbert space $H$.
	If $\Phi : X \rightarrow Y$ is a map between operator systems, we say that $\Phi$ is \emph{(u)cp}, \emph{cb}, \emph{cc}, \emph{(u)ci} if $\Phi$ is respectively (unital) completely positive, completely bounded, completely contractive, (unital) completely isometric.
	
	We occasionally refer to \emph{order-unit} spaces by which we mean a real partially ordered vector space $V$ with an Archimedean order unit, which furthermore induces a norm on $V$.
	
	\begin{notation}
		If $X$ is an operator system, we denote by $\mathbf{1}_X$ the unit in $X$.
		If $V$ is any vector space, we denote by $\mathbf{I}^V$ the identity map $V \rightarrow V$.
	\end{notation}
	
	Let $\Phi : X \rightarrow Y$ be a map between operator systems $X$ and $Y$.
	If $\Phi$ is positive, it is self-adjoint, i.e.\@ $\Phi(x^*) = \Phi(x)^*$, for all elements $x \in X$ \cite[Exercise 2.1]{Pau02}.
	If $\Phi$ is ucp, it is cc.
	If $\Phi$ is cb with $\lVert \Phi\rVert_\mathrm{cb} = \lVert\Phi(\mathbf{1}_X)\rVert$, the map $\Phi$ is cp.
	In particular, if the map $\Phi$ is uci, the maps $\Phi$ and $\Phi^{-1} : \Phi(X) \rightarrow Y$ are cp \cite[Proposition 3.5]{Pau02}.
	
	\begin{notation}\label{not:Spatial-tensor-product}
		For Hilbert spaces $H, K$ and subspaces $X \subseteq \mathcal{B}(H)$, $Y \subseteq \mathcal{B}(K)$, we denote by $X \otimes Y$ the \emph{spatial} tensor product, i.e.\@ the completion of the algebraic tensor product $X \odot Y$ in $\mathcal{B}(H \otimes K)$, where $H \otimes K$ is the Hilbert space tensor product.
	\end{notation}
	
	We will refer to the following result as the \emph{Fubini theorem} for cb/cp maps.
	
	\begin{lemma}\label{lem:Fubini-for-cb-maps}
		Let $X_1 \subseteq \mathcal{B}(H_1)$, $X_2 \subseteq \mathcal{B}(H_2)$ be operator spaces (respectively operator systems) and let $\Phi_1 : X_1 \rightarrow \mathcal{B}(K_1)$, $\Phi_2 : X_2 \rightarrow \mathcal{B}(K_2)$ be cb (respectively cp) maps.
		Then the map $\Phi_1 \odot \Phi_2 : X_1 \odot X_2 \rightarrow \mathcal{B}(K_1 \otimes K_2)$ extends uniquely to a cb (respectively cp) map $\Phi_1 \otimes \Phi_2 : X_1 \otimes X_2 \rightarrow \mathcal{B}(K_1 \otimes K_2)$ on the spatial tensor product such that $\lVert\Phi_1 \otimes \Phi_2\rVert_\mathrm{cb} \leq \lVert\Phi_1\rVert_\mathrm{cb} \lVert\Phi_2\rVert_\mathrm{cb}$.
		
		In particular, the following commutativity property holds:
		\begin{align}
			(\Phi_1 \otimes \mathbf{I}^{\mathcal{B}(H_2)}) (\mathbf{I}^{X_1} \otimes \Phi_2)
			= (\mathbf{I}^{\mathcal{B}(H_1)} \otimes \Phi_2) (\Phi_1 \otimes \mathbf{I}^{X_2})
			= \Phi_1 \otimes \Phi_2
		\end{align}
	\end{lemma}
	
	\begin{proof}
		For the first part, see \cite[Theorem 12.3]{Pau02}.
		The commutativity property then follows immediately from the commutativity property on the algebraic tensor products,
		\begin{align}
			(\Phi_1 \odot \mathbf{I}^{\mathcal{B}(H_2)}) (\mathbf{I}^{X_1} \odot \Phi_2)
			= (\mathbf{I}^{\mathcal{B}(H_1)} \odot \Phi_2) (\Phi_1 \odot \mathbf{I}^{X_2})
			= \Phi_1 \odot \Phi_2,
		\end{align}
		and from the existence and uniqueness of the extension to the spatial tensor product.
	\end{proof}
	
	In the special case that the maps $\Phi_1, \Phi_2$ in the above lemma are linear functionals, we refer to it as the Fubini theorem for slice maps which has been known since \cite{Tom67}.
	
	We point out that (unital) complete positivity is a property concerning arbitrary (not just matrix) amplifications of a map between operator systems:
	
	\begin{lemma}
		Let $X, Y$ be operator systems and let $\Phi : X \rightarrow Y$ be a unital linear map.
		Then the map $\Phi$ is cp if and only if the amplification $\Phi \otimes \mathbf{I}^Z : X \otimes Z \rightarrow Y \otimes Z$ is cp, for any operator system $Z$.
	\end{lemma}
	
	\begin{proof}
		By \cite[Corollary 5.1.2]{ER20} the unital map $\Phi$ is cp if and only if $\Phi$ is cc and by \cite[Proposition 2.1.1]{Pis03} this is equivalent to $\Phi \otimes \mathbf{I}^Z$ being cc, for any operator space $Z$.
		Applying \cite[Corollary 5.1.2]{ER20} again gives the equivalence with $\Phi \otimes \mathbf{I}^Z$ being cp.
	\end{proof}	
	
	\begin{definition}
		Let $Y \subseteq X$ be operator systems.
		A \emph{ucp conditional expectation} is a ucp map $E : X \rightarrow Y$ such that $E(y) = y$, for all $y \in Y$.
	\end{definition}
	
	In other words a ucp conditional expectation is an idempotent ucp map $E : X \rightarrow X$ with range $Y$.
	
	The \emph{state space} of an operator system $X$ is the set of positive linear functionals on $X$ of norm $1$, which we denote by $\mathcal{S}(X)$.
	
	Let $X \subseteq \mathcal{B}(H)$ be an operator system and, for a directed set $\mathcal{L}$, let $(P_\Lambda)_{\Lambda \in \mathcal{L}}$ be a net of orthogonal projections in $\mathcal{B}(H)$.
	For every $\Lambda \in \mathcal{L}$, set $H_\Lambda := P_\Lambda H$.
	Assume that the net $(P_\Lambda)_{\Lambda \in \mathcal{L}}$ is a \emph{join semilattice} for the relation of containment of ranges, i.e.\@ for all $\Lambda_1, \Lambda_2 \in \mathcal{L}$, the orthogonal projection $P_{\Lambda_1 \vee \Lambda_2}$ onto the closed subspace $H_{\Lambda_1} + H_{\Lambda_2}$ is in the net. 
	Assume furthermore that the net $(P_\Lambda)_{\Lambda \in \mathcal{L}}$  converges strongly to the identity $\mathbf{I}^H \in \mathbf{B}(H)$.
	Let $\tau_\Lambda : \mathcal{B}(H) \rightarrow \mathcal{B}(H_\Lambda)$ be the compression map, i.e.\@ $\tau_\Lambda(T) := P_\Lambda T P_\Lambda$, for all bounded operators $T \in \mathcal{B}(H)$, and denote by $X_\Lambda := \tau_\Lambda(X)$ the operator subsystem of $\mathcal{B}(H_\Lambda)$ given by the image of the operator system $X$ under $\tau_\Lambda$.
	We set $\mathcal{S}_\mathcal{L} := \bigcup_{\Lambda \in \mathcal{L}} \tau_\Lambda^* \mathcal{S}(X_\Lambda) \subseteq \mathcal{S}(X)$, where $\tau_\Lambda^* : \mathcal{S}(X_\Lambda) \rightarrow \mathcal{S}(X)$ is the pullback of the map $\tau_\Lambda$.
	
	\begin{lemma}[{\cite[Proposition 15]{GEvS23}}]\label{lem:Weak*-Density-Of-Liftable-States}
		The set $\mathcal{S}_\mathcal{L}$ is dense in the state space $\mathcal{S}(X)$ for the weak$^*$ topology.
	\end{lemma}
	
	\begin{proof}
		Observe that the set $\mathcal{S}_\mathcal{L}$ is convex.
		Indeed, every subset $\tau_\Lambda^* \mathcal{S}(X_\Lambda) \subseteq \mathcal{S}_\mathcal{L}$ is convex, since the pullback map $\tau_\Lambda^* : \mathcal{S}(X_\Lambda) \rightarrow \mathcal{S}(X)$ is affine.
		Now, observe that for closed subspaces $H_{\Lambda_1} \subseteq H_{\Lambda_2}$ we have that $P_{\Lambda_1} P_{\Lambda_2} = P_{\Lambda_2} P_{\Lambda_1} = P_{\Lambda_1}$, so that we can consider the restriction $\tau_{\Lambda_1}\big|_{X_{\Lambda_2}} (T) := P_{\Lambda_1} a P_{\Lambda_1}$, for all elements $T \in X_{\Lambda_2}$ and $a \in X$ with $\tau_{\Lambda_2}(a) = T$.
		In particular, since $\tau_\Lambda$ is onto, for all closed subspaces $H_\Lambda$, $\Lambda \in \mathcal{L}$, we have that $\tau_\Lambda^* : \mathcal{S}(X_\Lambda) \rightarrow \mathcal{S}(X)$ and $(\tau_{\Lambda_1}\big|_{X_{\Lambda_2}})^* : \mathcal{S}(X_{\Lambda_1}) \rightarrow \mathcal{S}(X_{\Lambda_2})$ are injections.
		Therefore, if $\phi \in \tau_{\Lambda_1}^*\mathcal{S}(X_{\Lambda_1}), \psi \in \tau_{\Lambda_2}^*\mathcal{S}(X_{\Lambda_2})$ are states, any convex combination $t\phi + (1-t)\psi$ (for $0 \leq t \leq 1$) is a state $t (\tau_{\Lambda_1}\big|_{X_{\Lambda_1 \vee \Lambda_2}})^* \phi + (1-t) (\tau_{\Lambda_2}\big|_{X_{\Lambda_1 \vee \Lambda_2}})^* \psi$ in $\tau_{\Lambda_1 \vee \Lambda_2}^*\mathcal{S}(X_{\Lambda_1 \vee \Lambda_2})$, which establishes convexity of $\mathcal{S}_\mathcal{L}$.
		
		Now, since the subspace $\sum_{H_\Lambda \in \mathcal{L}} H_\Lambda$ is dense in $H$ by strong convergence $P_\Lambda \rightarrow \mathbf{I}^H$, the set $\mathcal{S}_\mathcal{L}$ contains a dense subset $\mathcal{S}_{\mathcal{L},\mathrm{vec}}$ of the vector states on $X$, so that an element $x \in X$ is positive if the complex number $\rho(x)$ is positive, for all vector states $\rho \in \mathcal{S}_{\mathcal{L},\mathrm{vec}}$, and thus for all states $\rho \in \mathcal{S}_\mathcal{L}$.
		Therefore, by \cite[Theorem 4.3.9]{KR1}, the set $\mathrm{co}(\mathcal{S}_\mathcal{L}) = \mathcal{S}_\mathcal{L}$ is weak$^*$-dense in $\mathcal{S}(X)$ as claimed.
	\end{proof}
	\black
	
	For convenience, we record the following well-known consequence of the \emph{Kadison function representation}.
	
	\begin{lemma}\label{lem:Kadison-function-representation}
		For every element $x \in X$ of an operator system $X$, the following holds:
		\begin{align}
			\sup_{\phi \in \mathcal{S}(X)} |\phi(x)| 
			\leq \lVert x\rVert
			\leq 2 \sup_{\phi \in \mathcal{S}(X)} |\phi(x)|
		\end{align}
	\end{lemma}
	
	\begin{proof}
		The first inequality is immediate since states are positive functionals of norm $1$.
		Indeed, we have $x = \Re(x) + i\Im(x)$, where $\Re(x) = \frac{x+x^*}{2}$ and $\Im(x) = \frac{i(x^*-x)}{2}$ are self-adjoint, so that $\lVert\Re(x)\rVert = \sup_{\phi \in \mathcal{S}(X)} |\phi(\Re(x))|$ and similarly for $\Im(x)$ \cite[Theorem 4.3.9]{KR1}.
		The claim then follows by triangle inequality.
	\end{proof}

	\section{Compact quantum groups}\label{sec:Compact-quantum-groups}
	
	We consider compact quantum groups in the sense of Woronowicz \cite{Wor98} and summarize their main properties, which are most important for this article, following the exposition in \cite{NT13}.
	See also \cite{KS97} for another standard reference which, however, takes a more (Hopf $^*$-)algebraic approach.
	
	As in \autoref{not:Spatial-tensor-product}, for two C$^*$-algebras $A_1$, $A_2$, we denote by $A_1 \otimes A_2$ their minimal tensor product.
	
	\begin{definition}
		A \emph{compact quantum group} is a pair $(A,\Delta)$, where $A$ is a unital $\mathrm{C}^*$-algebra and $\Delta : A \rightarrow A \otimes A$ is the \emph{comultiplication map}, i.e.\@ a unital $^*$-homomorphism which is \emph{coassociative}, i.e.
		\begin{align}
			(\mathbf{I}^A \otimes \Delta)\Delta = (\Delta \otimes \mathbf{I}^A)\Delta,
		\end{align} 
		and such that the \emph{Podleś density} (or \emph{cancellation}) \emph{property} is satisfied:
		\begin{align}
			\overline{\mathrm{span}}((A \otimes \mathbf{1}_A) \Delta(A)) = A \otimes A = \overline{\mathrm{span}}((\mathbf{1}_A \otimes A) \Delta(A))
		\end{align}
	\end{definition}
	
	We think of the $\mathrm{C}^*$-algebra $A$ as the ``\emph{function algebra}'' $\mathrm{C}(\mathbb{G})$ of a (virtual) compact quantum group $\mathbb{G}$ and will (ab)use this terminology and notation throughout.
	
	\begin{notation}
		We use Sweedler notation for the comultiplication, i.e.\@ we set $a_{(0)} \otimes a_{(1)} := \Delta(a)$, for all $a \in A$.
		Coassociativity allows for an unambiguous use of the notations $(\mathbf{I}^A \otimes \Delta)\Delta(a) = a_{(0)} \otimes a_{(1)} \otimes a_{(2)} = a_{(-1)} \otimes a_{(0)} \otimes a_{(1)} := (\Delta \otimes \mathbf{I}^A)\Delta(a)$.
		If $\theta, \vartheta$ are any cb maps with domain $A$ we set $a_{(0)} \otimes \theta(a_{(1)}) := (\mathbf{I}^A \otimes \theta) \Delta(a)$, $\vartheta(a_{(0)})\otimes a_{(1)} := (\vartheta \otimes \mathbf{I}^A) \Delta(a)$ and, by \autoref{lem:Fubini-for-cb-maps}, we may unambiguously write $\vartheta(a_{(0)}) \otimes \theta(a_{(1)}) := (\vartheta \otimes \theta) \Delta(a)$.
		If $\theta$ or $\vartheta$ are functionals, we may of course omit the tensor product ``$\otimes$'' in this notation.
	\end{notation}
	
	Fix a compact quantum group $\mathbb{G}$ with function algebra $A = \mathrm{C}(\mathbb{G})$ and comulti\-pli\-cation $\Delta : A \rightarrow A \otimes A$.
	
	\begin{definition}
		A \emph{unitary (right) corepresentation} $\pi$ of the compact quantum group $\mathbb{G}$ is given by a Hilbert space $H_\pi$ and a unitary element $U^\pi \in \mathcal{M}(\mathcal{K}(H_\pi) \otimes A)$, such that 
		\begin{align}\label{eqn:Corepresentation-property}
			(\mathbf{I} \otimes \Delta) (U^\pi) = U^\pi_{12} U^\pi_{13}.
		\end{align}
		If the Hilbert space $H_\pi$ is finite dimensional, the corepresentation $\pi$ is called \emph{finite dimensional} and we set $\dim(\pi) := \dim(H_\pi)$.
	\end{definition}
	
	In (\ref{eqn:Corepresentation-property}) above $\mathbf{I} \otimes \Delta : \mathcal{M}(\mathcal{K}(H_\pi) \otimes A) \rightarrow \mathcal{M}(\mathcal{K}(H_\pi) \otimes A \otimes A)$ denotes the unique extension of the map $\mathbf{I}^{\mathcal{K}(H_\pi)} \otimes \Delta$ on $\mathcal{K}(H_\pi) \otimes A$.
	Recall that there are two canonical embeddings of $\mathcal{M}(\mathcal{K}(H_\pi) \otimes A)$ into $\mathcal{M}(\mathcal{K}(H_\pi) \otimes A \otimes A)$, which are given by the unique extensions of the maps $\mathcal{K}(H_\pi) \otimes A \rightarrow \mathcal{K}(H_\pi) \otimes A \otimes A$ defined by $T \otimes a \mapsto T \otimes a \otimes \mathbf{1}_A$ and $T \otimes a \mapsto T \otimes \mathbf{1}_A \otimes a$ respectively.
	The elements $U^\pi_{12}, U^\pi_{13} \in \mathcal{M}(\mathcal{K}(H_\pi) \otimes A \otimes A)$ denote the respective images of $U^\pi$ under these two canonical embeddings.
	See also \cite{MVD98} for more background on this definition.
	Note that in the finite dimensional case $\mathcal{M}(\mathcal{K}(H_\pi) \otimes A) = \mathcal{B}(H_\pi) \otimes A$.
	
	Every finite dimensional unitary corepresentation $\pi$ induces an \emph{isometric comodule map} $\delta^\pi : H_\pi \rightarrow H_\pi \otimes A$, given by $\delta^\pi(\xi) := U^\pi(\xi \otimes \mathbf{1}_A)$,
	where $H_\pi$ is identified with $\mathcal{B}(\mathbb{C},H_\pi)$.
	Being a comodule map means that $\delta^\pi$ satisfies the \emph{comodule property}
	\begin{align}\label{eqn:Comodule-property}
		(\mathbf{I}^H \otimes \Delta) \delta^\pi = (\delta^\pi \otimes \mathbf{I}^A) \delta^\pi,
	\end{align}
	and being \emph{isometric} means
	\begin{align}
		\delta^\pi(\xi)^*\delta^\pi(\eta) = \overline{\langle \xi, \eta \rangle_{H_\pi}} \mathbf{1}_A,
	\end{align}
	for all vectors $\xi, \eta \in H_\pi$, where the convention in this article is that Hilbert space inner products are antilinear in the second component.
	Conversely, every isometric comodule map $\delta : H \rightarrow H \otimes A$ on a finite dimensional Hilbert space $H$ gives rise to a finite dimensional unitary corepresentation \cite[Lemma 1.7]{DC16}.
	
	An \emph{intertwiner} of two finite dimensional unitary corepresentations $\pi, \rho$ is an operator $T : H_\pi \rightarrow H_\rho$ such that $(T \otimes \mathbf{1}_A) U^\pi = U^\rho (T \otimes \mathbf{1}_A)$.
	The set of all intertwiners of the corepresentations $\pi$ and $\rho$ is denoted by $\mathrm{Mor}(\pi,\rho)$.
	If the set $\mathrm{Mor}(\pi,\rho)$ contains a unitary element, the corepresentations $\pi$ and $\rho$ are called \emph{unitarily equivalent}.
	The set $\mathrm{End}(\pi) := \mathrm{Mor}(\pi,\pi)$ is a $\mathrm{C}^*$-algebra and the co\-re\-pre\-sen\-tation $\pi$ is called \emph{irreducible} if $\mathrm{End}(\pi) = \mathbb{C} \mathbf{I}^{H_\pi}$.
	\emph{Schur's lemma} states that two finite dimensional irreducible unitary corepresentations $\pi,\rho$ are either unitarily equivalent and $\mathrm{dim}(\mathrm{Mor}(\pi,\rho)) = 1$, or that $\dim(\mathrm{Mor}(\pi,\rho)) = 0$. 
	We denote the set of unitary equivalence classes of finite dimensional irreducible unitary corepresentations by $\widehat{\mathbb{G}}$.
	
	There is a unique left and right invariant state $h_A \in \mathcal{S}(A)$, i.e.\@ a state which satisfies
	\begin{align}
		a_{(0)}h_A(a_{(1)}) = h_A(a_{(0)})a_{(1)} = h_A(a) \mathbf{1}_A,
	\end{align}
	for all elements $a \in A$.
	The state $h_A \in \mathcal{S}(A)$ is called the \emph{Haar state} of the compact quantum group $\mathbb{G}$.
	
	Denote by $H = \mathrm{L}^2(\mathbb{G})$ the Hilbert space of the GNS-representation $\pi_{h_A} : A \rightarrow \mathcal{B}(H)$ of the $\mathrm{C}^*$-algebra $A$ induced by the Haar state $h_A$.
	Denote the GNS-map by $\Lambda : A \rightarrow H$.
	Assume moreover, that the $\mathrm{C}^*$-algebra $A$ is faithfully represented on a Hilbert space $H_0$ and denote the inclusion of $A$ into $\mathcal{B}(H_0)$ by $\iota$.
	There are two unitary operators $W \in \mathcal{M}(\mathcal{K}(H) \otimes A), V \in \mathcal{M}(A \otimes \mathcal{K}(H))$ which satisfy
	\begin{align}
		W(\Lambda(a) \otimes \xi) 
		&= (\pi_{h_A} \otimes \iota)(\Delta(a))(\Lambda(\mathbf{1}_A) \otimes \xi), \\
		V(\xi \otimes \Lambda(a)) &= (\iota \otimes \pi_{h_A})(\Delta(a))(\xi \otimes \Lambda(\mathbf{1}_A)),
	\end{align}
	for all elements $a \in A$ and $\xi \in H_0$.
	The unitaries $W$, $V$ define unitary (respectively right, left) corepresentations of the compact quantum group $\mathbb{G}$ and are usually referred to as the \emph{multiplicative unitaries}.
	In particular, they implement the comultiplication $\Delta$ as follows:
	\begin{align}\label{eqn:Multiplicative-Unitary}
		W(\pi_{h_A}(a) \otimes \mathbf{1}_{\mathcal{B}(H_0)})W^* &= (\pi_{h_A} \otimes \iota) \Delta(a), \\
		V(\mathbf{1}_{\mathcal{B}(H_0)} \otimes \pi_{h_A}(a))V^* &= (\iota \otimes \pi_{h_A}) \Delta(a),
	\end{align}
	for all elements $a \in A$.
	For more details, see \cite[Section 1.5]{NT13}, \cite[Section 11.3.6]{KS97}.
	
	For a finite dimensional unitary corepresentation $\pi$ and vectors $\xi, \eta \in H_\pi$, we denote by $\omega_{\xi,\eta}^\pi$ the functional on $\mathcal{K}(H_\pi)$ given by $T \mapsto \langle T\xi, \eta \rangle_{H_\pi}$.
	Then the elements $(\omega_{\xi,\eta}^\pi \otimes \mathbf{I}^A) (U^\pi) \in A$, for $\xi, \eta \in H_\pi$, are called the \emph{matrix coefficients} of the corepresentation $\pi$.
	Denote by $\mathcal{A} = \mathcal{O}(\mathbb{G})$ the linear span of all the matrix coeffients of all finite dimensional irreducible unitary corepresentations.
	The set $\mathcal{A}$ is a \emph{Hopf $^*$-algebra}, i.e.\@ a unital $^*$-algebra with a coassociative \emph{comultiplication} map $\Delta : \mathcal{A} \rightarrow \mathcal{A} \otimes \mathcal{A}$, an \emph{antipode} $S : \mathcal{A} \rightarrow \mathcal{A}$ and a \emph{counit} $\epsilon : \mathcal{A} \rightarrow \mathbb{C}$, which satisfy $S(a_{(0)})a_{(1)} = a_{(0)}S(a_{(1)}) = \epsilon(a) \mathbf{1}_\mathcal{A}$ and the \emph{counit property}
	\begin{align}
		\epsilon(a_{(0)}) a_{(1)} = a_{(0)} \epsilon(a_{(1)}) = a,
	\end{align}
	for all $a \in \mathcal{A}$.
	The Hopf $^*$-algebra $\mathcal{A}$ is dense in the $\mathrm{C}^*$-algebra $A$ and its unit and comultiplication are those inherited from $A$.
	We call $\mathcal{A}$ the \emph{coordinate algebra} of the compact quantum group $\mathbb{G}$.
	
	There is a quantum group version of \emph{Peter--Weyl theory} which is crucial for our purposes.
	It states that every unitary corepresentation decomposes into a direct sum of finite dimensional irreducible unitary corepresentations \cite[Theorem 1.5.4]{NT13}.
	For the multiplicative unitaries $W, V$ this gives an orthogonal decomposition of the GNS Hilbert space
	\begin{align}\label{eqn:Peter--Weyl-decomposition}
		H = \bigoplus_{\gamma \in \widehat{\mathbb{G}}} H_\gamma \otimes \overline{H_\gamma},
	\end{align}
	which is respected by the multiplicative unitaries.
	I.e.\@ $W$ and $V$ restrict to a unitary operators on $H_\gamma \otimes \overline{H_\gamma} \otimes H_0$ and $H_0 \otimes H_\gamma \otimes \overline{H_\gamma}$ respectively, for all finite dimensional irreducible unitary corepresentations $\gamma \in \widehat{\mathbb{G}}$.
	More concretely, for any finite dimensional irreducible unitary corepresenation $\gamma$, one can define a bilinear map $\beta : H_\gamma \times \overline{H_\gamma} \rightarrow H$, given by $(\xi,\overline{\eta}) \mapsto \dim_\mathrm{q}(\gamma)^{\nicefrac{1}{2}} (\omega_{\xi,Q_\gamma^{\nicefrac{1}{2}}  \eta}^\gamma \otimes \Lambda)(U^\gamma)$, \emph{cf.\@} the end of Section 1.5 of \cite{NT13}.
	Then, for a fixed vector $\overline{\eta} \in \overline{H_\gamma}$, the induced linear map $\ell^\gamma_{\overline{\eta}} := \beta(\cdot,\overline{\eta}) : H_\gamma \rightarrow H$ intertwines the corepresentations $\gamma$ and $W$.
	If $\overline{\eta} \in H_\gamma$ is a unit vector, the map $\ell^\gamma_{\overline{\eta}}$ is an isometry, and if the vectors $\overline{\eta}, \overline{\eta}' \in H_\gamma$ are orthogonal, so are the images of the maps $\ell^\gamma_{\overline{\eta}}$ and $\ell^\gamma_{\overline{\eta}'}$.
	The span of the images $\ell^\gamma_{\overline{\eta}}(\xi)$, for all vectors $\xi \in H_\gamma$, $\overline{\eta} \in \overline{H_\gamma}$ and all finite dimensional irreducible unitary corepresentations $\gamma \in \widehat{\mathbb{G}}$, in the GNS space $H$ is equal to the image of the coordinate algebra under the GNS map, so a dense subspace of $H$ \cite[Corollary 1.5.5]{NT13}.
	
	The function algebra $A$ of the compact quantum group $\mathbb{G}$ can come in different versions.
	On the one hand, the \emph{universal} function algebra $A_\mathrm{u} = \mathrm{C}_\mathrm{u}(\mathbb{G})$ is given by the universal $\mathrm{C}^*$-completion of the coordinate algebra $\mathcal{A}$ and the comultipliation $\Delta : \mathcal{A} \rightarrow \mathcal{A} \otimes \mathcal{A}$ extends to a $^*$-homomorphism $A_\mathrm{u} \rightarrow A_\mathrm{u} \otimes A_\mathrm{u}$ which is still coassociative and satisfies the Podleś density property.
	Also the counit $\epsilon : \mathcal{A} \rightarrow \mathbb{C}$ extends to a bounded map $A_\mathrm{u} \rightarrow \mathbb{C}$ satisfying the counit property.
	On the other hand, the \emph{reduced} function algebra $A_\mathrm{r} = \mathrm{C}_\mathrm{r}(\mathbb{G})$ is given by the image $\pi_{h_A}(A) \subseteq \mathcal{B}(H)$ under the GNS representation. 
	The cyclic vector $\xi_{h_A}$ for the GNS representation of the function algebra $A$ induces a bi-invariant state $\langle \cdot \xi_{h_A}, \xi_{h_A} \rangle$ on the reduced function algebra $A_\mathrm{r}$ which is still called the \emph{Haar state} and which turns out to be faithful.
	Moreover, the Haar state $h_A$ is faithful on the coordinate algebra $\mathcal{A}$, so that we may regard the reduced function algebra $A_\mathrm{r}$ as a completion of the coordinate algebra.
	In particular, the comultiplication map on $\mathcal{A}$ extends to $A_\mathrm{r}$.
	The universal and reduced function algebra come with $^*$-homomorphisms
	\begin{align}
		A_\mathrm{u} \overset{\pi_\mathrm{u}}{\longrightarrow} A \overset{\pi_\mathrm{r}}{\longrightarrow} A_\mathrm{r},
	\end{align}
	which extend the identity maps on the coordinate algebra $\mathcal{A}$.
	The existence of the $^*$-homomorphism $\pi_\mathrm{u}$ is guaranteed by universality of the $\mathrm{C}^*$-algebra $A_\mathrm{u}$ and the $^*$-homomorphism $\pi_\mathrm{r}$ is the GNS representation.
	
	In general, neither the counit $\epsilon : \mathcal{A} \rightarrow \mathbb{C}$ extends to a bounded map on the reduced function algebra $A_\mathrm{r}$, nor does the Haar state $h_A$ extend to a faithful state on the universal function algebra $A_\mathrm{u}$.
	It turns out, however, that both is true if and only if the $^*$-homomorphism $\pi_\mathrm{r} \circ \pi_\mathrm{u} : A_\mathrm{u} \rightarrow A_\mathrm{r}$ is an isomorphism \cite{BMT01}.
	In that case the compact quantum group $\mathbb{G}$ is called \emph{coamenable}.
	
	We end this section by noting that the dual $A^*$ can be given an algebra structure as follows:
	\begin{align}
		\mu \ast \nu (a) := (\mu \otimes \nu) \Delta(a),
	\end{align}
	for all functionals $\mu, \nu \in A^*$ and elements $a \in A$.
	This restricts to a semigroup structure on the state space $\mathcal{S}(A)$.
	If the counit $\epsilon : A \rightarrow \mathbb{C}$ is bounded (in particular, if compact quantum group $\mathbb{G}$ is coamenable), the counit is a state on $A$ and it is the unit for the convolution product $*$.

	\section{Coactions}\label{sec:Coactions}	
	
	Since we take an operator system point of view throughout, we collect the relevant notions of coactions on operator systems.
	Much of this is algebraic in nature, i.e.\@ it can be seen in the setting of coactions on order unit spaces.
	See \cite{Sai09} for this point of view, and \cite{DC16} for a thorough $\mathrm{C}^*$-algebraic treatment.
	
	For the theory of coactions on operator systems, we follow \cite{dRH24,dR23}.
	In this subsection, we fix a compact quantum group $\mathbb{G}$ and denote its reduced function algebra by $A := \mathrm{C}_\mathrm{r}(\mathbb{G})$ and the comultiplication by $\Delta : A \rightarrow A \otimes A$.
	We furthermore fix an operator system $X$.
	
	\begin{definition}\label{def:Action-on-Operator-System}
		A \emph{right coaction} $\alpha$ of the function algebra $A$ on the operator system $X$ is a uci map $\alpha : X \rightarrow X \otimes A$ such that the \emph{coaction property} 
		\begin{align}\label{eqn:Coaction-property}
			(\alpha \otimes \mathbf{I}^A) \alpha 
			= (\mathbf{I}^X \otimes \Delta) \alpha
		\end{align} 
		and the Podleś density condition 
		\begin{align}\label{eqn:Podles-density-property}
			\overline{\mathrm{span}}((\mathbf{1}_X \otimes A)\alpha(X)) = X \otimes A
		\end{align} 
		are satisfied.
		A \emph{left coaction} $\beta : X \rightarrow A \otimes X$ is defined analogously.
		
		We say that a right coaction $\alpha$ and a left coaction $\beta$ \emph{cocommute} if the following holds:
		\begin{align}\label{eqn:Cocommuting-coactions}
			(\beta \otimes \mathbf{I}^A) \alpha = (\mathbf{I}^A \otimes \alpha) \beta
		\end{align}
	\end{definition}
	
	\begin{notation}
		We use Sweedler notation whenever convenient, i.e.\@ for an element $x \in X$, we write 
		\begin{align}
			x_{(0)} \otimes x_{(1)} := \alpha(x) \in X \otimes A,
		\end{align}
		as well as
		\begin{align}
			x_{(0)} \otimes x_{(1)} \otimes x_{(2)} \in X \otimes A \otimes A,
		\end{align}
		for any of the two maps in (\ref{eqn:Coaction-property}) applied to $x$.
		Similarly,
		\begin{align}
			x_{(-1)} \otimes x_{(0)} \otimes x_{(1)} \in A \otimes X \otimes A,
		\end{align}
		for any of the two maps in (\ref{eqn:Cocommuting-coactions}) applied to $x$.
	\end{notation}
	
	\begin{remark}
		Coactions on an operator system generalize reduced coactions of the reduced function algebra on $\mathrm{C}^*$-algebras.
		Indeed, a reduced ($\mathrm{C}^*$-algebraic) coaction $\alpha : B_\mathrm{r} \rightarrow B_\mathrm{r} \otimes A_\mathrm{r}$, with $A_\mathrm{r}$ the reduced function algebra of the compact quantum group $\mathbb{G}$, is an injective $^*$-homomorphism \cite[Proposition 3.4]{Li09}.
		By \cite[Corollary II.2.2.9]{Bla06} it follows that the map $\alpha$ is an isometry and arguing similarly for the matrix amplifications $\alpha^{(n)} = \alpha \otimes \mathbf{1}_{\mathrm{M}_n(\mathbb{C})}$, it follows that $\alpha$ is uci.
		
		In particular, if the compact quantum group $\mathbb{G}$ is coamenable, every $\mathrm{C}^*$-algebraic coaction of its function algebra on a unital $\mathrm{C}^*$-algebra is a coaction in the operator system sense.
		The comultiplication $\Delta : A \rightarrow A \otimes A$ on the reduced function algebra $A$ is an example of both, a right and left coaction of $A$ on the operator system $A$.
		
		Conversely, if $\alpha : X \rightarrow X \otimes A$ is a coaction in the operator system sense and if $X$ is a unital $\mathrm{C}^*$-algebra, the map $\alpha$ is a $^*$-homomorphism \cite[Proposition 3.7]{dRH24} and hence a $\mathrm{C}^*$-algebraic coaction.
	\end{remark}
	
	For the remainder of this subsection, we fix a right coaction $\alpha : X \rightarrow X \otimes A$.
	
	\begin{remark}\label{rmk:Module-structure-on-states-from-coaction}
		The coaction $\alpha$ gives the dual $X^*$ the structure of a right module for the convolution algebra $(A^*,*)$, which we denote as follows:
		\begin{align}
			\phi \lhd \mu (x) := (\phi \otimes \mu) \alpha(x) = \phi(x_{(0)}) \mu(x_{(1)}),
		\end{align}
		for all elements $x \in X$ and functionals $\phi \in X^*$, $\mu \in A^*$. 
		The right action $\lhd$ restricts to a right action of the semigroup $(\mathcal{S}(A),*)$ on $\mathcal{S}(X)$.
		
		If $Y$ is another operator system with a coaction $\beta : Y \rightarrow Y \otimes A$ and if there is a ucp onto map $\Phi : X \rightarrow Y$ which is equivariant for the coaction $\alpha$ and $\beta$, i.e.
		\begin{align}
			(\Phi \otimes \mathbf{I}^A) \alpha = \beta \Phi,
		\end{align}
		then, for all states $\psi \in \mathcal{S}(Y)$, $\mu \in \mathcal{S}(A)$, we have
		\begin{align}\label{eqn:Pullback-convolved-state-Sweedler}
			\psi(\Phi(x_{(0)})) \mu(x_{(1)}) = \psi(\Phi(x)_{(0)}) \mu(\Phi(x)_{(1)}),
		\end{align}
		for all elements $x \in X$.
		In other words,
		\begin{align}
			(\Phi^*\psi) \lhd \mu = \Phi^*(\psi \lhd \mu) \in \mathcal{S}(X),
		\end{align}
		where $\Phi^* : Y^* \rightarrow X^*$ denotes the pullback map which restricts to a map between the state spaces.
	\end{remark}
	
	A convenient feature of the requirement that coactions on operator systems be uci maps (rather than just ucc) is the following:
	
	\begin{lemma}\label{lem:Counit-Property-Operator-System}
		Assume that the compact quantum group $\mathbb{G}$ is coamenable.
		Then the \emph{counit property} also holds for the coaction $\alpha$, i.e.
		\begin{align}
			(\mathbf{I}^X \otimes \epsilon) \alpha = \mathbf{I}^X,
		\end{align}
		or, in Sweedler notation,
		\begin{align}
			x_{(0)} \epsilon(x_{(1)}) = x,
		\end{align}
		for all elements $x \in X$.
	\end{lemma}
	
	\begin{proof}
		From the counit and coaction properties, together with the Fubini theorem \autoref{lem:Fubini-for-cb-maps}, we obtain
		\begin{align}
			\alpha (\mathbf{I}^X \otimes \epsilon) \alpha
			= (\mathbf{I}^X \otimes \mathbf{I}^X \otimes \epsilon) (\alpha \otimes \mathbf{I}^A) \alpha
			= (\mathbf{I}^X \otimes \mathbf{I}^X \otimes \epsilon) (\mathbf{I}^X \otimes \Delta) \alpha
			= \alpha.
		\end{align}
		Since the map $\alpha$ is uci, it is in particular injective, so the claim follows.
	\end{proof}
	
	\begin{definition}
		A \emph{fixed point} for the coaction $\alpha$ is an element $x \in X$ which satisfies
		\begin{align}
			x_{(0)} \otimes x_{(1)} = x \otimes \mathbf{1}_A.
		\end{align}
		The set of fixed points is denoted by $X^\alpha := \{x \in X \mid \alpha(x) = x \otimes \mathbf{1}_A\}$.
		
		The coaction $\alpha$ is called \emph{ergodic} if its only fixed points are multiples of the unit, i.e.\@ $X^\alpha = \mathbb{C}\mathbf{1}_X$.
	\end{definition}
	
	\begin{example}
		The coaction $\Delta : A \rightarrow A \otimes A$ is ergodic.
	\end{example}
	
	\begin{definition}
		Let $\pi$ be a finite dimensional unitary corepresentation of the compact quantum group $\mathbb{G}$.
		An \emph{intertwiner} of $\pi$ and the coaction $\alpha$ is a linear map $T : H_\pi \rightarrow X$ such that 
		\begin{align}
			\alpha T = (T \otimes \mathbf{I}^A) \delta^\pi,
		\end{align}
		where the map $\delta^\pi : H_\pi \rightarrow H_\pi \otimes A$ is the isometric comodule map associated to the corepresentation $\pi$ by $\delta^\pi(\xi) := U^\pi(\xi \otimes \mathbf{1}_A)$.
		The set of all intertwiners of $\pi$ and $\alpha$ is denoted by $\mathrm{Mor}(\pi,\alpha)$.
		
		For the corepresentation $\pi$, the \emph{isotypical component} is defined by
		\begin{align}
			X^\pi := \span\{T\xi \in X \mid T \in \mathrm{Mor}(\pi,\alpha), \xi \in H_\pi \} \subseteq X.
		\end{align}
		See \autoref{ex:Isotypical-component-trivial-corep} for consistency of the notation $X^\pi$ and $X^\alpha$.
		We denote the linear span of all isotypical components in $X$ of finite dimensional unitary corepresentations by
		\begin{align}
			\mathcal{X} := \sum_{\gamma \in \widehat{\mathbb{G}}} X^\gamma \subseteq X.
		\end{align}
		The set $\mathcal{X}$ is called the \emph{algebraic core} of the operator system $X$ for the coaction $\alpha$.
	\end{definition}	
	
	\begin{example}\label{ex:Isotypical-component-trivial-corep}
		The isotypical component $X^{\mathbf{1}}$ for the trivial corepresentation $\mathbf{1} = \mathbf{1}_A \in \mathcal{B}(\mathbb{C}) \otimes A$ coincides with the set of fixed points $X^\alpha$.
		Indeed, by definition, a linear map $T : \mathbb{C} \rightarrow X$ is an intertwiner of the corepresentation $\mathbf{1}$ and the coaction $\alpha$ if and only if 
		$\alpha (T (\lambda)) = (T \otimes \mathbf{I}^A)\mathbf{1}(\lambda)$, for all $\lambda \in H_\mathbf{1} = \mathbb{C}$, which we may rewrite as $(T(\lambda))_{(0)} \otimes (T(\lambda))_{(1)} =  T(\lambda) \otimes \mathbf{1}_A$.
		Hence $T \in \mathrm{Mor}(\mathbf{1},\alpha)$ if and only if $T(\lambda) \in X$ is a fixed point for the coaction $\alpha$.
	\end{example}
	
	\begin{definition}
		A state $\phi \in \mathcal{S}(X)$ is called \emph{invariant} for the coaction $\alpha$ if 
		\begin{align}
			\phi \lhd \mu = \mu(\mathbf{1}_A) \phi,
		\end{align}
		for all functionals $\mu \in A^*$.
	\end{definition}
	
	\begin{lemma}\label{lem:Existence-Invariant-State-for-Action}
		The following properties hold:
		\begin{enumerate}
			\item The set of fixed points $X^\alpha$ is an operator subsystem of $X$.
			\item The following map $E_\alpha : X \rightarrow X^\alpha$ is a ucp conditional expectation:
			\begin{align}
				E_\alpha(x) = x_{(0)} h_A(x_{(1)}),
			\end{align} 
			for all $x \in X$.
			\item If the coaction $\alpha$ is ergodic, the following defines an invariant state $h_X \in \mathcal{S}(X)$:
			\begin{align}\label{eqn:Invariant-State}
				E_\alpha(x) = h_X(x) \mathbf{1}_X,
			\end{align}
			for all $x \in X$.
			The state $h_X$ is the unique invariant state on $X$.
		\end{enumerate}
	\end{lemma}
	
	For parts (2) and (3) of the proof, we follow essentially the arguments in \cite[Lemma 4]{Boc95}.
	
	\begin{proof}
		(1)
		Clearly, the unit $\mathbf{1}_X$ is a fixed point for the coaction $\alpha$.
		Moreover, the set of fixed points $X^\alpha$ is self-adjoint, since $\alpha(x^*) = \alpha(x)^* = x^* \otimes \mathbf{1}_A$, for $x \in X^\alpha$.
		This shows that $X^\alpha$ is an operator system.
		
		(2)
		Note that $E_\alpha$ is a ucp map being the composition of the uci map $\alpha$ and the ucp map $\mathbf{I}^X \otimes h_A$.
		If $x \in X^\alpha$ is a fixed point, we have $E_\alpha(x) =x_{(0)} h_A(x_{(1)}) = x h(\mathbf{1}_A) = x\alpha$.
		For $x \in X$, we obtain by invariance of the Haar state $h_A$:
		\begin{align}
			\alpha(E_\alpha(x))
			&= \alpha \left( (\mathbf{I}^X \otimes h_A) \alpha (x) \right) \\
			&= x_{(0)} \otimes x_{(1)} h_A(x_{(2)}) \\
			&= x_{(0)} \otimes h_A(x_{(1)}) \mathbf{1}_A \\
			&= E_\alpha(x) \otimes \mathbf{1}_A,
		\end{align}
		which shows that $E_\alpha(X) \subseteq X^\alpha$.
		
		(3)
		By ergodicity, the range of $E_\alpha$ is $X^\alpha = \mathbb{C} \mathbf{1}_X$.
		In particular, the map $h_X : X \rightarrow \mathbb{C}$ is ucp, whence a state.
		We check that the state $h_X$ is invariant.
		To this end, let $\phi \in \mathcal{S}(X)$ be an arbitrary state and $\mu \in A^*$ a functional.
		Note that we then have $h_X(x) = \phi(E_\alpha(x)) = \phi(x_{(0)}) h_A(x_{(1)})$, for all $x \in X$.
		With this and using invariance of the Haar state $h_A$, we obtain
		\begin{align}
			(h_X \lhd \mu)(x)
			&= h_X(x_{(0)}) \mu(x_{(1)}) \\
			&= \phi(x_{(0)}) h_A(x_{(1)}) \mu(x_{(2)}) \\
			&= \phi(x_{(0)}) h_A(x_{(1)}) \mu(\mathbf{1}_A) \\
			&= h_X(x) \mu(\mathbf{1}_A)
		\end{align}
		for all $x \in X$.
		This shows invariance of the state $h_X$.
		
		To show uniqueness of the invariant state $h_X$, let $\phi \in \mathcal{S}(X)$ be another invariant state, i.e.\@ a state which satisfies $\phi(x_{(0)})\mu(x_{(1)}) = \phi(x)\mu(\mathbf{1}_A)$, for all functionals $\mu \in A^*$.
		Then the following holds:
		\begin{align}
			\phi(x)
			= \phi(x) h_A(\mathbf{1}_A)
			= \phi(x_{(0)}) h_A(x_{(1)})
			= \phi(E_\alpha(x))
			= h_X(x)
		\end{align}
		for all $x \in X$.
	\end{proof}
	
	\begin{lemma}\label{lem:Properties-isotypical-components}
		Let $\pi \in \widehat{\mathbb{G}}$ be a finite dimensional unitary corepresentation of $\mathbb{G}$.
		The following properties hold:
		\begin{enumerate}
			\item There is an idempotent map $E_\pi : X \rightarrow X^\pi$.
			\item The isotypical component $X^\pi$ is a closed subspace of $X$.
			\item The algebraic core $\mathcal{X}$ is a dense operator subsystem of $X$.
			\item The coaction $\alpha$ restricts to the isotypical component, i.e.
			\begin{align}
				\alpha(X^\pi) \subseteq X^\pi \otimes \mathcal{A}^\pi,
			\end{align}
			where $\mathcal{A}^\pi$ is the coalgebra of matrix coefficients of the corepresentation $\pi$.
			In particular, the coaction $\alpha$ restricts to a Hopf $^*$-algebra coaction $\alpha : \mathcal{X} \rightarrow \mathcal{X} \otimes \mathcal{A}$, i.e.\@ that in addition to having the \emph{coaction property}, the map $\alpha$ is \emph{$^*$-preserving} and \emph{counital}, i.e.\@ $x_{(0)}\epsilon(x_{(1)}) = x$, for all $x \in \mathcal{X}$ and where $\epsilon : \mathcal{A} \rightarrow \mathbb{C}$ is the counit of the Hopf $^*$-algebra $\mathcal{A}$.
			\item The ucp conditional expectation $E_\alpha : X \rightarrow X^\alpha$ is faithful, i.e., for all positive elements $x \in X_+$, if $E_\alpha(x) = 0$ it follows that $x = 0$. 
		\end{enumerate}
	\end{lemma}
	
	\begin{proof}
		Most of the claims are proven in \cite[Proposition 3.4]{dRH24}, see also \cite[Section 3]{DC16} for more details (with the apparent modifications for coactions on operator systems rather than $\mathrm{C}^*$-algebras).
		For the fact that the algebraic core $\mathcal{X}$ is an operator system, i.e.\@ unital and closed under the involution $^*$, we refer to \cite[Theorem 3.16]{DC16}. 
		
		To see (4), let $\xi \in H_\pi \cong \mathcal{B}(\mathbb{C},H_\pi)$ and $T \in \mathrm{Mor}(\pi,\alpha)$, and note that $\delta^\pi(\xi) = U^\pi(\xi \otimes \mathbf{1}_A) \in H_\pi \otimes A$ can be canonically identified with the linear map $H_\pi^* \ni \eta^* \mapsto (\eta^* \otimes \mathbf{1}_A) U^\pi (\xi \otimes \mathbf{1}_A) = (\omega^\pi_{\eta,\xi} \otimes \mathbf{I}^A) (U^\pi) \in \mathcal{A}^\pi$.
		It follows that $\alpha (T\xi) = (T \otimes \mathbf{I}^A) \delta^\pi(\xi) \in X^\pi \otimes \mathcal{A}^\pi$.
		
		The proof of faithfulness of the ucp conditional expectation $E_\alpha$ is as in \cite[Lemma 3.19]{DC16}.
	\end{proof}
	
	We point out that the coalgebra $\mathcal{A}^\pi$ of matrix coefficients of the corepresentation $\pi$ coincides with the isotypical component $A^\pi$ for the coaction $\Delta$ of $A$ on itself.
	Since we will not use this fact we refer to the remarks below \cite[Definition 3.13]{DC16} for a proof.
	
	The following lemma is an important tool for our later arguments. 
	The proof of the ergodicity statement is inspired by the proof of \cite[Proposition 9]{GEvS23}.
	
	\begin{lemma}\label{lem:Induced-action}
		Let $\tau : X \rightarrow Y$ be a ucp map onto an operator system $Y$.
		Assume that there is a well-defined uci map $\alpha^\tau : Y \rightarrow Y \otimes A$ such that the following holds:
		\begin{align}\label{eqn:Induced-action}
			(\tau \otimes \mathbf{I}^A) \alpha = \alpha^\tau \tau
		\end{align}
		Then the map $\alpha^\tau$ is a coaction on the operator system $Y$.
		Moreover, if the coaction $\alpha$ is ergodic, so is the coaction $\alpha^\tau$.
	\end{lemma}
	
	\begin{notation}
		We write (\ref{eqn:Induced-action}) in Sweedler notation as
		\begin{align}\label{eqn:Induced-action-Sweedler}
			\tau(x_{(1)}) \otimes x_{(2)} 
			= (\tau(x))_{(1)} \otimes (\tau(x))_{(2)}
			= y_{(1)} \otimes y_{(2)} \in Y \otimes A,
		\end{align}
		for all elements $x \in X$ and $y \in Y$ with $\tau(x) = y$.
	\end{notation}
	
	\begin{proof}
		The coaction property for $\alpha^\tau$ readily follows from that for $\alpha$:
		\begin{align}
			(\alpha^\tau \otimes \mathbf{I}^A) \alpha^\tau \tau
			&= (\alpha^\tau \otimes \mathbf{I}^A) (\tau \otimes \mathbf{I}^A) \alpha \\
			&= (\tau \otimes \mathbf{I}^A \otimes \mathbf{I}^A) (\alpha \otimes \mathbf{I}^A) \alpha \\
			&= (\tau \otimes \mathbf{I}^A \otimes \mathbf{I}^A) (\mathbf{I}^A \otimes \Delta) \alpha  \\
			&= (\mathbf{I}^Y \otimes \Delta) (\tau \otimes \mathbf{I}^A) \alpha \\
			&= (\mathbf{I}^Y \otimes \Delta) \alpha^\tau \tau
		\end{align}
		
		Similarly, the Podleś density property $\overline{\mathrm{span}}((\mathbf{1}_Y \otimes A) \alpha^\tau(Y)) = Y \otimes A$ readily follows from that of $\alpha$.
		Indeed the span of elements of the form
		\begin{align}
			(\mathbf{1}_Y \otimes a) \alpha^\tau(\tau(x))
			= (\mathbf{1}_Y \otimes a) \left((\tau \otimes \mathbf{I}^A) (\alpha(x))\right) 
			= (\tau \otimes \mathbf{I}^A) \left((\mathbf{1}_A \otimes a) \alpha(x)\right),
		\end{align}	
		with $a \in A$, $x \in X$, is dense in $Y \otimes A$, since $\tau$ is onto.
		
		We assume now that the coaction $\alpha$ is ergodic.
		Recall from \autoref{lem:Existence-Invariant-State-for-Action} that there is a unique state $h_X$ on $X$ which is invariant for the coaction $\alpha$, and which can be defined by $h_X(x)\mathbf{1}_X = x_{(0)} h_A(x_{(1)})$, for all elements $x \in X$.
		For any fixed point $y \in Y^{\alpha^\tau}$, i.e.\@ which satisfies $\alpha^\tau(y) = y \otimes \mathbf{1}_A$, and any element $x \in X$ with $\tau(x) = y$, the following holds:
		\begin{align}\label{eqn:Fixed-points-of-induced-action-are-scalars}
			y
			= \tau(x)
			= \tau(x) h_A(\mathbf{1}_A)
			= \tau(x_{(0)}) h_A(x_{(1)})
			= \tau(h_X(x)\mathbf{1}_X)
			= h_X(x)\mathbf{1}_Y
		\end{align}
		Therefore, the fixed point $y$ is an element of $\mathbb{C}\mathbf{1}_Y$ and thus the induced coaction $\alpha^\tau$ is ergodic.
	\end{proof}

	\section{Compact quantum metric spaces}\label{sec:CQMS}
	
	\subsection{Lip-norms}
	
	The ideas of \emph{Lip-norms} and \emph{compact quantum metric spaces} go back at least to \cite{Rie98, Rie99, Rie04} where they were developed for order unit spaces.
	We work exclusively in the setting of operator systems.
	
	\begin{definition}
		Let $X$ be an operator system.
		By a \emph{seminorm} $L$ on $X$ we always understand an extended seminorm $L : X \rightarrow [0,\infty]$.
		A seminorm $L : X \rightarrow [0,\infty]$ is called a \emph{Lipschitz seminorm} if it satisfies the following properties:
		\begin{enumerate}
			\newcounter{counter}
			\item It has \emph{dense domain}, i.e.\@ $\mathrm{Dom}(L) := \{x \in X \mid L(x) < \infty\}$ is dense in $X$,
			\item it is \emph{*-invariant}, i.e.\@ $L(x^*) = L(x)$, for all $x \in X$,
			\item it is $0$ on scalars, i.e.\@ $\mathbb{C}\mathbf{1}_X \subseteq \ker(L)$.
			\setcounter{counter}{\value{enumi}}
		\end{enumerate}
		A Lipschitz seminorm $L$ is called a \emph{Lip-norm} if additionally the following property holds:
		\begin{enumerate}
			\setcounter{enumi}{\value{counter}}
			\item The induced \emph{Monge--Kantorovich} distance 
			\begin{align}
				d^{L}(\phi,\psi) := \sup \{|\phi(x)-\psi(x)| \mid L(x) \leq 1\}
			\end{align}
			metrizes the weak* topology on the state space $\mathcal{S}(X)$.
		\end{enumerate}
	\end{definition}
	
	We emphasize the equivalence of the order unit space and operator system approach \cite[Proposition 2.8]{KK22}.
	
	\begin{remark}
		Let $X$ be an operator system with connected state space $\mathcal{S}(X)$.
		Then the kernel of any Lip-norm $L$ on $X$ is actually equal to $\mathbb{C} \mathbf{1}_X$ \cite[Lemma 2.2]{KK21}.
		Indeed, any two states $\phi, \psi \in \mathcal{S}(X)$ must be at finite distance from each other, since $d^L$ metrizes the weak$^*$-topology, for which $\mathcal{S}(X)$ is compact (and thus has finite diameter as a metric space).
		But since states on $X$ separate points (in the sense that if $\phi(x) = 0$, for all $\phi \in \mathcal{S}(X)$, it follows that $x = 0$ \cite[Theorem 4.3.4(i)]{KR1}), for every $x \in X \setminus \mathbb{C}\mathbf{1}_X$, there must be two states $\phi, \psi \in \mathcal{S}(X)$ such that $\phi(x) \neq \psi(x)$.
		If we now assume that $L(x)=0$ we have that $d^L(\phi,\psi) \geq |\phi(\lambda x) - \psi(\lambda x)|$, for all $\lambda \in \mathbb{C}$, by definition of the Monge--Kantorovich distance.
		Hence, $d^L(\phi,\psi) = \infty$ which contradicts the assumption that $d^L$ metrizes the compact set $\mathcal{S}(X)$.
	\end{remark}
	
	\begin{definition}
		A \emph{compact quantum metric space} is an operator system equipped with a Lip-norm.
	\end{definition}
	
	\begin{definition}
		A \emph{morphism} between two compact quantum metric spaces $(X,L_X)$ and $(Y,L_Y)$ is a ucp map $\Phi : X \rightarrow Y$, for which there is a constant $C \geq 0$ such that $L_Y(\Phi(x)) \leq C L_X(x)$, for all $x \in X$.
		A morphism $\Phi$ is called \emph{Lip-norm contractive} if $L_Y(\Phi(x)) \leq L_X(x)$, for all $x \in X$.
	\end{definition}
	
	\begin{definition}
		Let $X$ be an operator system and let $L : X \rightarrow [0,\infty]$ be a Lipschitz seminorm.
		Denote by $\lVert\cdot\rVert_{\sfrac{X}{\mathbb{C}}}$ and $L_{\sfrac{X}{\mathbb{C}}}$ the induced norm and seminorm on the quotient $\sfrac{X}{\mathbb{C}\mathbf{1}_X}$ respectively.
		The \emph{radius} of $X$ is the number $r_X := \inf\{r \in [0,\infty] \mid \lVert\cdot\rVert_{\sfrac{X}{\mathbb{C}}} \leq r  L_{\sfrac{X}{\mathbb{C}}}\}$.
	\end{definition}
	
	When working in the operator system setting of compact quantum metric spaces it might come to surprise that the Monge--Kantorovich distances of Lip-norms are only required to metrize the weak$^*$ topology without any requirement on the matrix state spaces.
	Indeed, a Lipschitz seminorm gives rise to Monge--Kantorovich distances on the matrix state spaces (see below) and it turns out that in the case of a Lip-norm, the diameters of all the matrix state spaces coincide \cite[Proposition 2.9]{Ker03} and the Monge--Kantorovich distances metrize the point-norm topologies on all matrix state spaces \cite[Proposition 2.12]{Ker03}.
	
	\begin{definition}
		Let $X$ be an operator system and $n \in \mathbb{N}$ be a positive integer.
		We denote by $\mathcal{S}_n(X)$ the set of ucp maps $X \rightarrow M_n(\mathbb{C})$ and call it a \emph{matrix state space} of $X$.
		If $L$ is a Lipschitz seminorm on $X$, we denote by 
		\begin{align}
			d^{L,n}(\phi,\psi) := \sup_{x \in X \setminus \ker(L_X)} \frac{\lVert\phi(x) - \psi(x)\rVert}{L(x)}
		\end{align}
		the induced \emph{Monge--Kantorovich} distance on $\mathcal{S}_n(X)$.
	\end{definition}
	
	\begin{notation}
		For a vector space $V$, (extended) seminorms $p, q$ on $V$ and any positive real number $r > 0$, we set $\mathrm{B}_r^p := \{ v \in V \mid p(v) < r\}$, $\overline{\mathrm{B}}_r^p := \{ v \in V \mid p(v) \leq r\}$, $\mathrm{B}_r^{p,q} := \mathrm{B}_r^p \cap \mathrm{B}_r^q$ and $\overline{\mathrm{B}}_r^{p,q} := \overline{\mathrm{B}}_r^p \cap \overline{\mathrm{B}}_r^q$.
	\end{notation}
	
	The following characterization of Lip-norms appears in \cite[Theorem 1.8]{Rie98}, see also \cite[Theorem 6.3]{Pav98} for an even earlier version in the $\mathrm{C}^*$-algebraic context.
	We state it here as in \cite[Proposition 2.5]{Li09}.
	Recall that a subset $S$ of a normed space $(V,\lVert\cdot\rVert)$ is \emph{totally bounded} if, for every $\varepsilon > 0$, there exist $v_1,\dots,v_n \in S$ such that $S \subseteq \bigcup_{i=1}^n \overline{\mathrm{B}}_\varepsilon^{\lVert\cdot\rVert}$.
	
	\begin{proposition}\label{lem:Total-boundedness}
		Let $X$ be an operator system and $L : X \rightarrow [0,\infty]$ a Lipschitz seminorm.
		Then $L$ is a Lip-norm if and only if $X$ has finite radius and the set $\overline{\mathrm{B}}_1^{\lVert\cdot\rVert,L}$ is totally bounded.
	\end{proposition}
	
	Note that if $X$ is a finite dimensional operator system, every Lipschitz seminorm $L$ on $X$ with $\ker(L) = \mathbb{C}\mathbf{1}_X$ is a Lip-norm.
	Indeed, the condition on the kernel of $L$ guarantees that $X$ has finite radius and by compactness the set $\overline{\mathrm{B}}_1^{\lVert\cdot\rVert,L}$ is totally bounded.

	\subsection{Gromov--Hausdorff type distances}
	
	Let $(X,L_X)$ and $(Y,L_Y)$ be compact quantum metric spaces.
	Different generalizations of Gromov--Hausdorff distance for compact metric spaces to the quantum setting have been proposed.
	The notion of classical Gromov--Hausdorff distance between the state spaces of compact quantum metric spaces was used e.g.\@ in \cite{CvS21, vS21}.
	
	\begin{definition}
		The \emph{Gromov--Hausdorff distance} of $(X,L_X)$ and $(Y,L_Y)$ is the classical Gromov--Hausdorff distance between their state spaces, i.e.\@
		\begin{equation}\label{eqn:GH-dist}
			\begin{aligned}
				\mathrm{dist}_{\mathrm{GH}}((X,L_X),(Y,L_Y))
				&:= \mathrm{dist}_{\mathrm{GH}}((\mathcal{S}(X),d^{L_X}),(\mathcal{S}(Y),d^{L_Y})) \\
				&\;= \inf\{\mathrm{dist}_\mathrm{H}^\rho(\mathcal{S}(X),\mathcal{S}(Y))\},
			\end{aligned}
		\end{equation} 
		where the infimum runs over all metrics $\rho$ on the disjoint union $\mathcal{S}(X) \sqcup \mathcal{S}(Y)$ which restrict to the respective Monge--Kantorovich metrics on the summands and where $\mathrm{dist}_\mathrm{H}^\rho$ is the usual Hausdorff distance \cite[Definition 7.3.1]{BBI01} on the set of compact subsets of the space $\mathcal{S}(X) \sqcup \mathcal{S}(Y)$ equipped with such a metric.
	\end{definition}
	
	The following notion of quantum Gromov--Hausdorff distance goes back to \cite{Rie04} where it was still phrased in terms of the order-unit space version of compact quantum metric spaces.
	We follow here the treatment in \cite[Section 2]{KK22}.
	For a Lipschitz seminorm $L$ on an operator system $X$ we denote by $L_{\mathrm{sa}}$ the restriction to the self-adjoint part $X_{\mathrm{sa}}$.
	
	\begin{definition}\label{def:Admissible-Lip-norm}
		A Lip-norm $L$ on the direct sum $X \oplus Y$ is called \emph{admissible} if $\Dom(L) = \Dom(L_X) \oplus \Dom(L_Y)$ and if the induced seminorms on $X_\mathrm{sa}$ and $Y_\mathrm{sa}$ under the respective coordinate projections $p_X : \Dom(L)_\mathrm{sa} \rightarrow \Dom(L_X)_\mathrm{sa}$ and $p_Y : \Dom(L)_\mathrm{sa} \rightarrow \Dom(L_Y)_\mathrm{sa}$ coincide with $(L_X)_\mathrm{sa}$ and $(L_Y)_\mathrm{sa}$ respectively.
	\end{definition}
	
	\begin{definition}
		The \emph{quantum Gromov--Hausdorff distance} of $(X,L_X)$ and $(Y,L_Y)$ is given by
		\begin{align}
			\mathrm{dist}_{\mathrm{GH}}^\mathrm{q}((X,L_X),(Y,L_Y))
			:= \inf \mathrm{dist}_\mathrm{H}^{d^L}(\mathcal{S}(X),\mathcal{S}(Y)),
		\end{align}
		where the infimum runs over all admissible Lip-norms $L$ on $X \oplus Y$ with $d^L$ the induced Monge--Kantorovich distance on $\mathcal{S}(X \oplus Y)$ and where $\mathrm{dist}_\mathrm{H}^{d^L}$ is the Hausdorff distance on the set of compact subsets of $\mathcal{S}(X \oplus Y)$.
	\end{definition}
	
	By \cite[Lemma 2.12]{KK22}, the quantum Gromov--Hausdorff distance of the two (operator system) compact quantum metric spaces $(X,L_X)$ and $(Y,L_Y)$ coincides with that of the two (order unit) compact quantum metric spaces $(\Dom(L_X)_\mathrm{sa}, (L_X)_\mathrm{sa})$ and $(\Dom(L_Y)_\mathrm{sa}, (L_Y)_\mathrm{sa})$.
	
	\begin{remark}\label{rmk:Comparison-GH-qGH-distance}
		Note that if $L$ is an admissible Lip-norm on $X \oplus Y$, the metric $d^L$ metrizes the weak* topology on $\mathcal{S}(X \oplus Y) \cong \mathcal{S}(X) \sqcup \mathcal{S}(Y)$.
		Thus $d^L$ is an element of the set of metrics $\rho$ over which the infimum in the definition (\ref{eqn:GH-dist}) of $\mathrm{dist}_\mathrm{GH}$ is taken, so that we always have
		\begin{align}
			\mathrm{dist}_\mathrm{GH} \leq \mathrm{dist}_\mathrm{GH}^\mathrm{q}.
		\end{align}
		However, in general the Gromov--Hausdorff distance $\mathrm{dist}_\mathrm{GH}$ and the quantum Gromov--Hausdorff distance $\mathrm{dist}_\mathrm{GH}^\mathrm{q}$ are inequivalent metrics on the set of compact quantum metric spaces \cite{KK23}.
	\end{remark}
	
	The infimum for the Gromov--Hausdorff distance can equivalently be taken over all isometric embeddings $\iota_{\mathcal{S}(X)} : \mathcal{S}(X) \hookrightarrow \mathcal{T}$ and $\iota_{\mathcal{S}(Y)} : \mathcal{S}(Y) \hookrightarrow \mathcal{T}$ into compact metric spaces $(\mathcal{T},\rho)$ \cite[Definition 7.3.10, Remark 7.3.12]{BBI01}.
	
	Similarly, \cite[Lemma 2.12]{KK22} and \cite[Proposition 4.7]{Li06} imply the following:
	\begin{align}
		&\mathrm{dist}_\mathrm{GH}^\mathrm{q}((X,L_X),(Y,L_Y)) \\
		&= \mathrm{dist}_\mathrm{GH}^\mathrm{q}((\Dom(L_X)_\mathrm{sa},(L_X)_\mathrm{sa}), (\Dom(L_Y)_\mathrm{sa},(L_Y)_\mathrm{sa})) \\
		&= \inf \mathrm{dist}_\mathrm{H}^V \left(\overline{\mathrm{B}}_1^{(L_X)_\mathrm{sa}}, \overline{\mathrm{B}}_1^{(L_Y)_\mathrm{sa}}\right),
	\end{align}
	where the infimum is taken over all order unit spaces $V$ which contain $\Dom(L_X)_\mathrm{sa}$ and $\Dom(L_Y)_\mathrm{sa}$ as order-unit subspaces.
	Equivalently, one can take the infimum over all normed spaces $V$ which contain $\Dom(L_X)_\mathrm{sa}$ and $\Dom(L_Y)_\mathrm{sa}$ as subspaces such that their order units coincide.
	
	The fact that it is enough to consider embeddings into order unit spaces rather than operator systems hints at the defect of the quantum Gromov--Hausdorff distance not to capture the complete order structure of the involved operator systems.
	To overcome this, the notions of \emph{complete} \cite{Ker03} and \emph{operator Gromov--Hausdorff distance} \cite{Li06} were introduced.
	For both of these notions we follow the unified treatment in \cite{KL07}.
	
	\begin{definition}\label{def:Complete-GH-distance}
		The \emph{complete Gromov--Hausdorff distance} is given by
		\begin{align}
			\mathrm{dist}^\mathrm{s}((X,L_X),(Y,L_Y))
			:= \inf \sup_{n \in \mathbb{N}} \mathrm{dist}_\mathrm{H}^{d^{L,n}}(\mathcal{S}_n(X),\mathcal{S}_n(Y)),
		\end{align}
		where the infimum is taken over all admissible Lip-norms $L$ on $X \oplus Y$.
	\end{definition}
	
	\begin{remark}
		There is a slight subtlety about the notion of admissible Lip-norm.
		Namely, in the definition of complete Gromov--Hausdorff distance given in \cite{KL07} only admissible Lip-norms on $(X \oplus Y)_\mathrm{sa}$ (with the analogous notions of Lip-norm and admissibility for order unit spaces) are taken into account, which are furthermore required to be closed.
		
		First, we point out that for the definition of quantum and complete Gromov--Hausdorff distance it is equivalent to consider admissible Lip-norms on $X \oplus Y$ (in the operator system sense of \autoref{def:Admissible-Lip-norm}) and admissible Lip-norms on $(X \oplus Y)_\mathrm{sa}$ (in the order unit sense), as can be seen from the proof of \cite[Lemma 2.12]{KK22}.
		Indeed, any admissible Lip-norm $L^0$ in the order unit sense induces a Lip-norm $L^0_\mathrm{os}$ which is admissible in the operator system sense and, conversely, any admissible Lip-norm $L$ in the operator system sense restricts to a seminorm $L_\mathrm{sa} := L|_{(X \oplus Y)_\mathrm{sa}}$ which is an admissible Lip-norm in the order unit sense.
		
		Second, the question whether closedness should be required for admissible Lip-norms is discussed in \cite[p.\@ 319]{Li06}.
		Recall that a Lip-norm $L$ on an order unit space $V$ is called \emph{closed} if it coincides with its closure given by $\overline{L}(v) := \inf \{r \geq 0 \mid v \in r \cdot \mathrm{cl}(\overline{\mathrm{B}}_1^L)\}$, where $\mathrm{cl}(\overline{\mathrm{B}}_1^L)$ denotes the closure of $\overline{\mathrm{B}}_1^L$ in the completion of $V$.
		If $L$ is lower semicontinuous, it is not hard to see that the Monge--Kantorovich distances induced by $L$ and $\overline{L}$ coincide \cite[Proposition 4.4]{Rie99}.
		Moreover, by \cite[Theorem 4.2]{Rie99}, there is a largest lower semicontinuous Lip-norm $L_{d^L}$ which is bounded above by $L$ (namely, the Lipschitz constant on $\mathrm{C}(\mathcal{S}(V))$ for the distance $d^L$), and such that the Monge--Kantorovich distances induced by $L_{d^L}$ and $L$ on $\mathcal{S}(V)$ coincide.
		In particular, this shows that for quantum and complete Gromov--Hausdorff distance it is equivalent require closedness of admissible Lip-norms or not.
	\end{remark}
	
	\begin{remark}\label{rmk:Comparison-complete-qGH-distance}
		In \cite[Definition 3.2]{Ker03}, the $n$-distance is defined as
		\begin{align}
			\mathrm{dist}_n^\mathrm{s} ((X,L_X),(Y,L_Y))
			:= \inf \mathrm{dist}_\mathrm{H}^{d^{L,n}}(\mathcal{S}_n(X),\mathcal{S}_n(Y)),
		\end{align}
		where, again, the infimum is taken over all admissible Lip-norms $L$ on $X \oplus Y$.
		Moreover, the author points out that 
		\begin{align}
			\mathrm{dist}_m^\mathrm{s} 
			\leq \mathrm{dist}_n^\mathrm{s}
			\leq \mathrm{dist}^\mathrm{s},
		\end{align}
		for all $m, n \in \mathbb{N}$ with $m \leq n$.
		In particular it follows that
		\begin{align}
			\mathrm{dist}_\mathrm{GH}^\mathrm{q}
			= \mathrm{dist}_1^\mathrm{s}
			\leq \mathrm{dist}^\mathrm{s}.
		\end{align}
	\end{remark}
	
	The following distance is an operator system analogue of the order-unit space version given in \cite{Li06}.
	
	\begin{definition}[\cite{KL07}]
		The \emph{operator Gromov--Hausdorff distance} is given by
		\begin{align}
			\mathrm{dist}^\mathrm{op}((X,L_X),(Y,L_Y))
			:= \inf \mathrm{dist}_\mathrm{H}(\iota_X(\overline{\mathrm{B}}_1^{L_X}), \iota_Y(\overline{\mathrm{B}}_1^{L_Y})),
		\end{align}
		where the infimum is taken over all unital complete order embeddings $\iota_X : X \hookrightarrow Z$, $\iota_Y : Y \hookrightarrow Z$ into an operator system $Z$.
	\end{definition}
	
	It turns out that the complete and operator Gromov--Hausdorff distances are equal.
	
	\begin{proposition}[{\cite[Theorem 3.7]{KL07}}]\label{prop:Equality-complete-operator-GH-distance}
		The following holds:
		\begin{align}
			\mathrm{dist}^\mathrm{s}
			= \mathrm{dist}^\mathrm{op}
		\end{align}
	\end{proposition}

	\subsection{Criterion for the control of complete Gromov--Hausdorff distance}
	
	We record the following estimate of Hausdorff distance. 
	
	\begin{lemma}\label{lem:Hausdorff-distance-estimate}
		Let $M, N$ be compact subsets of a metric space $(Z,d)$ and let $f : M \rightarrow N$, $g : N \rightarrow M$ be set maps.
		Then the following holds:
		\begin{align}
			\mathrm{dist}_\mathrm{H}^d(M,N) 
			\leq \max \left\{ \sup_{m \in M} d(m,f(m)), \sup_{n \in N} d(g(n),n) \right\}
		\end{align}
	\end{lemma}
	
	\begin{proof}
		This is immediate from the definition of Hausdorff distance \cite[Definition 7.3.1]{BBI01}.
		For a subset $S \subseteq Z$ and a positive real number $r > 0$, set $U_r^d(S) := \{z \in Z \mid \inf_{s \in S} d(z,s) < r\}$.
		\begin{align}
			&\mathrm{dist}_\mathrm{H}^d(M,N) \\
			&= \inf \{r > 0 \mid M \subseteq U_r^d(N) \text{ and } N \subseteq U_r^d(M)\} \\
			&= \inf \left\{ r > 0 \mid \left(\inf_{n \in N} d(m,n) < r, \forall m \in M\right) \text{ and } \left(\inf_{m \in M} d(m,n) < r, \forall n \in N\right) \right\}  \\
			&\leq \inf \left\{ r > 0 \mid \left(d(m,f(m)) < r, \forall m \in M\right) \text{ and } \left(d(g(n),n) < r, \forall n \in N\right) \right\}  \\
			&= \max \left\{ \sup_{m \in M} d(m,f(m)), \sup_{n \in N} d(g(n),n) \right\}
		\end{align}
	\end{proof}
	
	The following sufficient condition for estimating complete Gromov--Hausdorff distance is analogous to the criteria \cite[Proposition 4]{vS21v1} and \cite[Proposition 2.14]{KK22}.
	They on the methods already used in \cite{Rie98,Rie99} and formalized in terms of the notion of \emph{bridge} in \cite{Rie04}.
	
	\begin{proposition}\label{prop:Criterion-qGH-comparison}
		Let $(X,L_X)$ and $(Y,L_Y)$ be compact quantum metric spaces and let $\varepsilon_X, \varepsilon_Y, C_\Phi, C_\Psi > 0$ be positive real numbers.
		Suppose that there are morphisms $\Phi : X \rightarrow Y$, $\Psi : Y \rightarrow X$ of compact quantum metric spaces with $L_Y(\Phi(x)) \leq C_\Phi L_X(x)$ and $L_X(\Psi(y)) \leq C_\Psi L_Y(y)$, for all $x \in X$, $y \in Y$.
		Assume furthermore that
		\begin{align}
			\lVert\Psi\Phi(x)-x\rVert \leq \varepsilon_X L_X(x) \quad \text{and} \quad
			\lVert\Phi\Psi(y)-y\rVert \leq \varepsilon_Y L_Y(y).
		\end{align}
		Then the following estimate holds:
		\begin{align}
			\mathrm{dist}^\mathrm{s}(X,Y) 
			\leq \max\left\{\mathrm{diam}(X,L_X) \left| 1-\frac{1}{C_\Phi}\right| + \frac{\varepsilon_X}{C_\Phi}, \mathrm{diam}(Y,L_Y) \left| 1-\frac{1}{C_\Psi}\right| + \frac{\varepsilon_Y}{C_\Psi}\right\}
		\end{align}
	\end{proposition}
	
	\begin{proof}
		We set $r := \max\left\{\mathrm{diam}(X,L_X) \left| 1-\frac{1}{C_\Phi}\right| + \frac{\varepsilon_X}{C_\Phi}, \mathrm{diam}(Y,L_Y) \left| 1-\frac{1}{C_\Psi}\right| + \frac{\varepsilon_Y}{C_\Psi}\right\}$ and define a seminorm $L$ on $X \oplus Y$ by 
		\begin{align}
			L(x,y) := \max \left\{ L_X(x), L_Y(y), \frac{1}{r} \lVert y - \Phi(x) \rVert, \frac{1}{r} \lVert x - \Psi(y)\rVert \right\}.
		\end{align}
		It is shown in the proof of \cite[Proposition 2.14]{KK22} that $L$ is an admissible Lip-norm.
		
		Observe that, for every positive integer $n \in \mathbb{N}$ and matrix state $\phi \in \mathcal{S}_n(X)$, we have
		\begin{align}
			d^{L,n}(\phi, \Psi^*\phi)
			&= \sup_{(x,y) \in X \oplus Y \setminus \mathbb{C}\mathbf{1}_{X \oplus Y}} \frac{\lVert\iota_{\mathcal{S}_n(X)}(\phi)(x,y) - \iota_{\mathcal{S}_n(Y)}(\Psi^*\phi)(x,y)\rVert}{L(x,y)} \\
			&= \sup_{(x,y) \in X \oplus Y \setminus \mathbb{C}\mathbf{1}_{X \oplus Y}} \frac{\lVert\phi(x) - \phi(\Psi(y))\rVert}{\max \left\{ L_X(x), L_Y(y), \frac{1}{r} \lVert y - \Phi(x)\rVert, \frac{1}{r} \lVert x - \Psi(y)\rVert \right\}} \\
			&\leq r,
		\end{align}
		where $\iota_{\mathcal{S}_n(X)} : \mathcal{S}_n(X) \rightarrow \mathcal{S}_n(X) \sqcup \mathcal{S}_n(Y)$, $\iota_{\mathcal{S}_n(Y)} : \mathcal{S}_n(Y) \rightarrow \mathcal{S}_n(X) \sqcup \mathcal{S}_n(Y)$ are the respective inclusion maps of matrix state spaces into the disjoint union of matrix state spaces.
		Similarly, for every positive integer $n \in \mathbb{N}$ and matrix state $\psi \in \mathcal{S}_n(Y)$, we have $d^{L,n}(\Phi^*\psi,\psi) \leq r$.
		
		Now, using the Lipschitz maps $\Psi^* : \mathcal{S}_n(X) \rightarrow \mathcal{S}_n(Y)$, $\Phi^* : \mathcal{S}_n(Y) \rightarrow \mathcal{S}_n(X)$ on the subsets $\mathcal{S}_n(X)$, $\mathcal{S}_n(Y)$ of the metric space $(\mathcal{S}_n(X) \sqcup \mathcal{S}_n(Y), d^{L,n})$, we obtain from \autoref{lem:Hausdorff-distance-estimate} that
		\begin{align}
			\mathrm{dist}_\mathrm{H}^{d^{L,n}}(\mathcal{S}_n(X),\mathcal{S}_n(Y))
			\leq \max \left\{ \sup_{\phi \in \mathcal{S}_n(X)} d^{L,n}(\phi, \Psi^*\phi), \sup_{\psi \in \mathcal{S}_n(Y)} d^{L,n}(\Phi^*\psi, \psi) \right\}.
		\end{align}
		Together with the previous observation that each of these two suprema is bounded by $r$, the claim follows:
		\begin{align}
			\mathrm{dist}^\mathrm{s}(X,Y)
			\leq \sup_{n \in \mathbb{N}} \mathrm{dist}_\mathrm{H}^{d^{L,n}}(\mathcal{S}_n(X),\mathcal{S}_n(Y))
			\leq r
		\end{align}
	\end{proof} 
	
	From \autoref{rmk:Comparison-GH-qGH-distance}, \autoref{rmk:Comparison-complete-qGH-distance} and \autoref{prop:Equality-complete-operator-GH-distance} the following corollary is immediate.
	
	\begin{corollary}\label{cor:Complete-GH-distance-bounds-all-others}
		Under the hypothesis of \autoref{prop:Criterion-qGH-comparison}, the distances $\mathrm{dist}_\mathrm{GH}$, $\mathrm{dist}_\mathrm{GH}^\mathrm{q}$, $\mathrm{dist}_n^\mathrm{s}$, for all $n \in \mathbb{N}$, $\mathrm{dist}^\mathrm{s}$ and $\mathrm{dist}^\mathrm{op}$ are all bounded above by the number $r$ from (the proof of) \autoref{prop:Criterion-qGH-comparison}.
	\end{corollary}

	\section{Invariant Lip-norms}\label{sec:Invariant-Lip-norms}
	
	Throughout this section, let $\mathbb{G}$ be a compact quantum group with reduced function algebra $A := \mathrm{C}_\mathrm{r}(\mathbb{G})$.
	Denote the comultiplication on $A$ by $\Delta$ and fix a right coaction $\alpha : X \rightarrow X \otimes A$.
	We use Sweedler notation throughout.
	Much of the terminology and results presented in this section are due to \cite{Li09} in the setting of coactions on $\mathrm{C}^*$-algebras, whereas here we consider coactions on operator systems.
	We point out that in loc.cit.\@ a right coaction is considered as a \emph{left $\mathbb{G}$-action}, so the reader has to make the according adjustments in terminology when referring back.
	
	\begin{definition}
		We say that a seminorm  $L_X : X \rightarrow [0,\infty]$ is (right) \emph{invariant} for the right coaction $\alpha$ if
		\begin{align}
			L_X(x_{(0)}\mu(x_{(1)})) \leq L_X(x), 
		\end{align}
		for all elements $x \in X$ and states $\mu \in \mathcal{S}(A)$.	
		(Left) \emph{invariance} for a left coaction is defined analogously.
		
		Similarly, a seminorm $L_A : A \rightarrow [0,\infty]$ is called \emph{right} (respectively \emph{left}) \emph{invariant} if it is invariant for the right (respectively left) coaction $\Delta$.
		The seminorm $L_A$ is called \emph{bi-invariant} if it is both right and left invariant.
	\end{definition} 
	
	\begin{example}
		For a compact group $G$ with a left invariant metric $d$, i.e.\@ $d(gh,g'h) = d(g,g')$, for all elements $g,g',h \in G$, the Lipschitz constant $\mathrm{Lip}$ is a right invariant Lip-norm on the $\mathrm{C}^*$-algebra of continuous functions on the group $G$ (with domain the Lipschitz functions on $G$). 
		See \cite{GEvS23}.
		
		Conversely, assume that $L$ is a Lip-norm on $\mathrm{C}(G)$, which is invariant for the right coaction $\mathrm{C}(G) \ni f \mapsto \big( (g,h) \mapsto f(gh) =: \rho_h(f)(g) \big) \in \mathrm{C}(G \times G) \cong \mathrm{C}(G) \otimes \mathrm{C}(G)$.
		Then the induced Monge--Kantorovich distance $d^L$ is left invariant for the action of $G$ on the state space $\mathcal{S}(\mathrm{C}(G))$ given by pullback of $\rho$, i.e.\@ $\rho_h^*\mu(f) := \mu(\rho_h(f))$.
		Indeed, by right invariance of $L$, for all elements $g \in G$, it holds that $L(f) \leq 1$ if and only if $L(\rho_{g^{-1}}(f)) \leq 1$.
		Therefore, $d^L(\rho_g^*\mu, \rho_g^*\nu)
		= \sup_{L(f) \leq 1} |\rho_g^*\mu(f) - \rho_g^*\nu(f)|
		= \sup_{L(\rho_{g^{-1}}(f)) \leq 1} |\mu(f) - \nu(f)|
		= \sup_{L(f) \leq 1} |\mu(f) - \nu(f)|
		= d^L(\mu,\nu)$.
	\end{example}
	
	\begin{definition}
		A Lipschitz seminorm $L_X$ on $X$ is called \emph{regular} if $L_X$ is finite on the dense operator subsystem $\mathcal{X} := \bigoplus_{\gamma \in \widehat{\mathbb{G}}} X^\gamma \subseteq X$.
	\end{definition}
	
	The following proposition is the main result of \cite{Li09}, where it is treated for coactions on unital $\mathrm{C}^*$-algebras \cite[Theorem 1.4]{Li09}.
	See \cite[Section 2.5]{Sai09} for an order unit space version.
	All arguments adapt to operator systems.
	We find it convenient to split the statements into the first three simple observations and the last main result.
	
	\begin{proposition}\label{prop:Li-Thm-1.4}
		Assume that the function algebra $A$ is equipped with a seminorm $L_A$.
		For all elements $x \in X$, set
		\begin{align}\label{eqn:Induced-Lip-norm}
			L_X^\alpha(x) 
			:= \sup_{\phi \in \mathcal{S}(X)} L_A(\phi(x_{(0)})x_{(1)}).
		\end{align}
		
		The following properties hold:
		\begin{enumerate}
			\item The function $L_X^\alpha : X \rightarrow [0,\infty]$ is a seminorm on $X$.
			\item If $L_A$ is a regular Lipschitz seminorm, so is the induced  seminorm $L_X^\alpha$.
			\item If the seminorm $L_A$ is right invariant, the induced seminorm $L_X^\alpha$ is invariant for the right coaction $\alpha$.
			\item Assume that the compact quantum group $\mathbb{G}$ is coamenable.
			If the seminorm $L_A$ is a regular Lip-norm and the coaction $\alpha$ is ergodic, the induced seminorm $L_X^\alpha$ is a Lip-norm.
		\end{enumerate}
	\end{proposition}
	
	\begin{proof}
		(1)
		The fact that $L_X^\alpha$ is a seminorm is immediate from the seminorm properties of $L_A$.
		
		(2)
		The fact that the seminorm $L_X^\alpha$ is $^*$-invariant follows from $^*$-invariance of the seminorm $L_A$ together with the identity $\phi((x^*)_{0})(x^*)_{(1)} = (\phi \otimes \mathbf{I}^A)\alpha(x^*) = ((\phi \otimes \mathbf{I}^A)\alpha(x))^* = \phi(x_{(0)})(x_{(1)})^*$, for all $x \in X$, $\phi \in \mathcal{S}(X)$.
		Moreover, since slice maps are unital, i.e.\@ $\phi((\mathbf{1}_X)_{(0)}) (\mathbf{1}_X)_{(1)} = \mathbf{1}_A$, it is clear that $\mathbb{C}\mathbf{1}_X \subseteq \ker(L_X^\alpha)$.
		Last, observe that since the coaction $\alpha$ restricts to a Hopf algebra coaction $\mathcal{X} \rightarrow \mathcal{X} \otimes \mathcal{O}(\mathbb{G})$, we have $\phi(x_{(0)})x_{(1)} \in \mathcal{O}(\mathbb{G})$, for all $x \in \mathcal{X}$.
		By regularity of $L_A$, we have $\mathcal{O}(\mathbb{G}) \subseteq \Dom(L_A)$.
		We conclude that $L_X^\alpha$ is finite on $\mathcal{X}$ and thus a regular Lipschitz seminorm.
		
		(3)
		Right invariance of the seminorm $L_X^\alpha$ is a direct computation:
		\begin{align}\label{eqn:Right-invariance}
			L_X^\alpha(x_{(0)}\mu(x_{(1)}))
			&= \sup_{\phi \in \mathcal{S}(X)} L_A(\phi(x_{(0)})x_{(1)}\mu(x_{(2)})) \\
			&\leq \sup_{\phi \in \mathcal{S}(X)} L_A(\phi(x_{(0)})x_{(1)}) \\
			&= L_X^\alpha(x),
		\end{align}
		for all elements $x \in X$ and states $\mu \in \mathcal{S}(A)$, where we applied the Fubini theorem for slice maps and right invariance of $L_A$. 
		
		(4)		
		To establish that $L_X^\alpha$ is a Lip-norm it remains to show that $(X,L_X^\alpha)$ has finite radius and that the subset $\overline{\mathrm{B}}_1^{\lVert\cdot\rVert,L_X^\alpha} \subseteq X$ is totally bounded.
		We refrain from going through the entire argument here, but point to \cite[Section 8]{Li09}, in particular Lemma 8.5, Lemma 8.6 and Lemma 8.7 therein, and \cite[Section 2.5]{Sai09} for details.
		However, we can deduce our claim from the results in \cite{Sai09}.
		In fact, by \cite[Lemma 2.19]{Sai09}, the order unit quantum metric space $(X_\mathrm{sa},(L_X^\alpha)_\mathrm{sa})$ has radius at most $2 r_{A_\mathrm{sa}}$, where $r_{A_\mathrm{sa}}$ is the radius of $(A_\mathrm{sa},(L_A)_\mathrm{sa})$.
		Applying the triangle inequality to the decomposition of $x$ into its real and imaginary part yields that the radius of $(X,L_X^\alpha)$ is at most $4r_{A_\mathrm{sa}}$.
		Moreover, it follows from the proof of \cite[Proposition 2.18]{Sai09} that the subset $\overline{\mathrm{B}}_1^{\lVert\cdot\rVert_{\mathrm{sa}},(L_X^\alpha)_\mathrm{sa}}$ of $X_{\mathrm{sa}}$ is totally bounded, from which we conclude that the closed subset $\overline{\mathrm{B}}_1^{\lVert\cdot\rVert,L_X^\alpha}$ of the totally bounded subset $\overline{\mathrm{B}}_1^{\lVert\cdot\rVert_{\mathrm{sa}},(L_X^\alpha)_\mathrm{sa}} + i\overline{\mathrm{B}}_1^{\lVert\cdot\rVert_{\mathrm{sa}},(L_X^\alpha)_\mathrm{sa}}$ of the operator system $X_\mathrm{sa}+iX_\mathrm{sa} = X$ is totally bounded.
	\end{proof}
	
	\begin{remark}\label{rmk:Making-Lip-norms-invariant}
		Any seminorm $L_A$ on the function algebra $A$ can be upgraded to a right invariant seminorm.
		Indeed, as in \cite[Proposition 8.9]{Li09}, we set
		\begin{align}
			L_A'(a) := \sup_{\mu \in \mathcal{S}(A)} L_A (a_{(0)}\mu(a_{(1)})),
		\end{align}
		and check for right invariance:
		\begin{align}
			L_A'(a_{(0)}\nu(a_{(0)}))
			&= L_A'((\mathbf{I}^A \otimes \nu) \Delta(a)) \\
			&= \sup_{\mu \in \mathcal{S}(A)} L_A((\mathbf{I}^A \otimes \mu) \Delta (\mathbf{I}^A \otimes \nu) \Delta (a)) \\
			&= \sup_{\mu \in \mathcal{S}(A)} L_A((\mathbf{I}^A \otimes \mu \otimes \nu) (\Delta \otimes \mathbf{I}^A) \Delta (a)) \\
			&= \sup_{\mu \in \mathcal{S}(A)} L_A(a_{(0)} (\mu \ast \nu) (a_{(1)})) \\
			&\leq \sup_{\mu \in \mathcal{S}(A)} L_A(a_{(0)} \mu (a_{(1)})) \\
			&= L_A'(a),
		\end{align}
		for all $a \in A$, $\nu \in \mathcal{S}(A)$.
		
		Similarly, setting
		\begin{align}
			L_A''(a) := \sup_{\mu \in \mathcal{S}(A)} L_A (\mu(a_{(0)})a_{(1)})
		\end{align}
		and 
		\begin{align}
			L_A''' := \max\{L_A',L_A''\}
		\end{align}
		give respectively left and bi-invariant seminorms on $A$.
	\end{remark}
	
	\begin{remark}\label{rmk:Invariant-seminorms-coamenable-case}
		If the compact quantum group $\mathbb{G}$ is coamenable, it is clear that $L_A \leq L_A'$, where $L_A'$ is the induced seminorm from \autoref{rmk:Making-Lip-norms-invariant}.
		Indeed, 
		\begin{align}
			L_A(a) = L_A(a_{(0)}\epsilon(a_{(1)})) \leq \sup_{\mu \in \mathcal{S}(A)} L_A(a_{(0)}\mu(a_{(1)}))
			= L_A'(a),
		\end{align}
		for all elements $a \in A$.
		Conversely, if $L_A$ is right invariant to begin with, we have that $L_A' \leq L_A$, so that in this case $L_A = L_A'$.
		
		Analogous statements hold for the induced left, respectively bi-invariant seminorms $L_A''$ and $L_A'''$.
	\end{remark}
	
	\begin{remark}\label{rmk:Bi-invariant-regular-Lip-norms-exist}
		As pointed out in \cite[Remark 8.2]{Li09}, it follows from \cite[Proposition 1.1]{Rie02a} that, if the function algebra $A$ is separable, it admits a regular Lip-norm.
		Together with \autoref{rmk:Making-Lip-norms-invariant} this shows that, if $A$ is the function algebra of a coamenable compact quantum group and if $A$ is separable, it admits a bi-invariant regular Lip-norm \cite[Corollary 8.10]{Li09}.
	\end{remark}
	
	For similar observations as the following, \emph{cf.\@} also the proofs of \cite[Lemma 8.7]{Li09} and \cite[Proposition 14]{GEvS23}.
	
	\begin{proposition}\label{lem:Lip-estimate-for-action-OS}
		Let $L_A$ be a Lipschitz seminorm on $A$ with $\ker(L_A) = \mathbb{C}\mathbf{1}_A$, and let $L_X^\alpha$ be the induced seminorm (\ref{eqn:Induced-Lip-norm}) on $X$.
		Let $\mu, \nu \in \mathcal{S}(A)$ be states on $A$ and consider the induced slice maps $X \rightarrow X$, given by $x \mapsto x_{(0)}\mu(x_{(1)})$ and $x \mapsto x_{(0)}\nu(x_{(1)})$ respectively. 
		Then the following holds, for all $x \in X$:
		\begin{align}
			\lVert x_{(0)}\mu(x_{(1)}) - x_{(0)}\nu(x_{(1)})\rVert \leq 2 d^{L_A}(\mu,\nu) L_X^\alpha(x) 
		\end{align}
		
		Similarly, if $\beta : X \rightarrow A \otimes X$ is a left coaction and $L_X^\beta$ the induced seminorm on $X$, the following holds, for all $x \in X$:
		\begin{align}
			\lVert \mu(x_{(-1)})x_{(0)} - \nu(x_{(-1)})x_{(0)}\rVert \leq 2 d^{L_A}(\mu,\nu) L_X^\beta(x)
		\end{align}
	\end{proposition}
	
	\begin{proof}
		For all elements $x \in X$ and any functional $\rho \in X^*$, the following holds:
		\begin{align}\label{eqn:Estimate-action-of-states}
			\lVert x_{(0)}\rho(x_{(1)})\rVert
			\leq 2 \sup_{\phi \in \mathcal{S}(X)} |\phi(x_{(0)}\rho(x_{(1)}))| 
			= 2 \sup_{\phi \in \mathcal{S}(X)} |\rho(\phi(x_{(0)})x_{(1)})|,
		\end{align}
		by the Kadison function representation and the Fubini theorem for slice maps. 
		Recall from the definition of the Monge--Kantorovich distance that $\frac{|(\mu-\nu)(a)|}{L_A(a)} \leq d^{L_A}(\mu,\nu)$, for all $a \in A \setminus \ker(L_A) = A \setminus \mathbb{C}\mathbf{1}_A$, and therefore \begin{align}
			|\mu(a)-\nu(a)| \leq d^{L_A}(\mu,\nu) L_A(a),
		\end{align} for all $a \in A$.
		By applying (\ref{eqn:Estimate-action-of-states}), the definition of the Monge--Kantorovich distance and the definition of the seminorm $L_X^\alpha$, we now obtain the result:
		\begin{align}
			\lVert x_{(0)}\mu(x_{(1)}) - x_{(0)}\nu(x_{(1)})\rVert
			&\leq 2 \sup_{\phi \in \mathcal{S}(X)} |(\mu-\nu)(\phi(x_{(0)})x_{(1)})| \\
			&\leq 2 \sup_{\phi \in \mathcal{S}(X)} d^{L_A}(\mu,\nu) L_A(\phi(x_{(0)})x_{(1)}) \\
			&= 2 d^{L_A}(\mu,\nu) L_X^\alpha(x) 
		\end{align}
		
		The proof of the statement for the left coaction $\beta$ is analogous.
	\end{proof}
	
	With the right and left coactions $\alpha$ and $\beta$ respectively replaced by the comultiplication $\Delta$, and the seminorms $L_X^\alpha$ and $L_X^\beta$ respectively replaced by the seminorms $L_A''$ and $L_A'$ from \autoref{rmk:Making-Lip-norms-invariant}, we obtain the following corollary.
	
	\begin{corollary}\label{cor:Lip-estimate-for-action-Alg}
		Assume that $L_A$ is a Lipschitz seminorm on $A$ with
		$\ker(L_A) = \mathbb{C}\mathbf{1}_A$, and let $L_A'$ and $L_A''$ be the induced right and left invariant seminorms as in \autoref{rmk:Making-Lip-norms-invariant}. 
		Then, for all states $\mu, \nu \in \mathcal{S}(A)$ and elements $a \in A$, the following inequalities hold:
		\begin{align}
			\lVert a_{(0)}\mu(a_{(1)}) - a_{(0)}\nu(a_{(1)})\rVert &\leq 2 d^{L_A}(\mu,\nu) L_A''(a), \text{ and } \\
			\lVert \mu(a_{(-1)})a_{(0)} - \nu(a_{(-1)})a_{(0)}\rVert &\leq 2 d^{L_A}(\mu,\nu) L_A'(a).
		\end{align}
	\end{corollary}

	\section{Peter--Weyl truncations of a compact quantum group}\label{sec:PW-truncations}
	
	We now investigate Peter--Weyl truncations of compact quantum groups and their convergence as compact quantum metric spaces.
	The reader will notice many analogies in the methods presented here and those used for Fourier truncations of compact quantum groups \cite{Rie23}.
	In fact, our exposition should set the stage for appropriately relating these two perspectives in future research.
	
	Let $\mathbb{G}$ be a compact quantum group and denote by $\widehat{\mathbb{G}}$ its set of unitary equivalence classes of finite dimensional unitary corepresentations.
	Write $A := \mathrm{C}_\mathrm{r}(\mathbb{G})$ for the reduced function algebra, $\Delta := \Delta_\mathrm{r}$ for the comultiplication thereon and $H := \mathrm{L}^2(\mathbb{G},h_A)$ for the GNS-space.
	Throughout this subsection fix a subset $\Lambda \subseteq \widehat{\mathbb{G}}$.
	This gives a closed subspace $H_\Lambda := \bigoplus_{\gamma \in \Lambda} H_\gamma \otimes \overline{H_\gamma}$ of the Hilbert space $H$ in the Peter--Weyl decomposition (\ref{eqn:Peter--Weyl-decomposition}).
	
	Denote by $P_\Lambda \in \mathcal{B}(\mathrm{L}^2(\mathbb{G}))$ the orthogonal projection onto the subspace $H_\Lambda$.
	The multiplicative unitaries $W, V \in \mathcal{B}(H \otimes H)$ commute with the projections $P_\Lambda \otimes \mathbf{I}^H$ and $\mathbf{I}^H \otimes P_\Lambda$ in $\mathcal{B}(H \otimes H)$ respectively, since $W((H_\gamma \otimes \overline{H_\gamma}) \otimes H) \subseteq (H_\gamma \otimes \overline{H_\gamma}) \otimes H$ and $V(H \otimes (H_\gamma \otimes \overline{H_\gamma})) \subseteq H \otimes (H_\gamma \otimes \overline{H_\gamma})$, for all $\gamma \in \widehat{\mathbb{G}}$.
	
	\begin{definition}
		We denote by $\tau_\lambda : \mathcal{B}(H) \rightarrow \mathcal{B}(H_\Lambda)$ the \emph{compression map}, given by
		\begin{align}
			\tau_\Lambda(T) := P_\Lambda T P_\Lambda,
		\end{align}
		for all $T \in \mathcal{B}(H)$, and write $A^{(\Lambda)} := \tau_\Lambda(A) \subseteq \mathcal{B}(H_\Lambda)$ for the image of the function algebra $A$ under the compression map.
	\end{definition}
	
	\begin{notation}
		Throughout this section we may drop the subindex $\Lambda$ of the projection $P_\Lambda$ and the compression map $\tau_\Lambda$ whenever convenient.
	\end{notation}
	
	\begin{remark}\label{rmk:Compression-map-ucp-onto}
		Note that $A^{(\Lambda)}$ is an operator system and the compression map $\tau : A \rightarrow A^{(\Lambda)}$ is ucp onto.
	\end{remark}
	
	\begin{proposition}\label{lem:Induced-Action-on-Toeplitz-System}
		There are unique ergodic cocommuting right and left coactions $\alpha^\tau : A^{(\Lambda)} \rightarrow A^{(\Lambda)} \otimes A$ and $\beta^\tau : A^{(\Lambda)} \rightarrow A \otimes A^{(\Lambda)}$ which satisfy
		\begin{align}\label{eqn:Induced-coactions-on-compression}
			(\tau \otimes \mathbf{I}^A) \Delta = \alpha^\tau\tau \text{ and } (\mathbf{I}^A \otimes \tau) \Delta = \beta^\tau\tau,
		\end{align}
		respectively.
	\end{proposition}
	
	\begin{proof}
		The claim follows from \autoref{lem:Induced-action} once we know that (\ref{eqn:Induced-coactions-on-compression}) well-defines maps $\alpha^{\tau}$ and $\beta^\tau$ which are furthermore uci.
		To this end, let $n \geq 1$ be a positive integer and let $(a_{ij})_{i,j} \in \mathrm{M}_n(A) \subseteq \mathcal{B}(H \otimes \mathbb{C}^n)$ be an $n \times n$ matrix with entries in $A$.
		Then, from the fact that the multiplicative unitary $W$ commutes with the projection $P \otimes \mathbf{I}^H$ and from unitarity of $W$, we obtain:
		\begin{align}
			\lVert \left((\tau \otimes \mathbf{I}^A)\Delta(a_{ij})\right)_{i,j}\rVert
			&= \lVert \left((P \otimes \mathbf{I}^H) W (a_{ij} \otimes \mathbf{1}_A) W^* (P \otimes \mathbf{I}^H)\right)_{i,j}\rVert \\
			&= \lVert (W \otimes \mathbf{I}^{\mathbb{C}^n}) \left((P \otimes \mathbf{I}^H) (a_{ij} \otimes \mathbf{1}_A) (P \otimes \mathbf{I}^H)\right)_{i,j} (W^* \otimes \mathbf{I}^{\mathbb{C}^n})\rVert \\
			&= \lVert \left((P \otimes \mathbf{I}^H) (a_{ij} \otimes \mathbf{1}_A) (P \otimes \mathbf{I}^H)\right)_{i,j} \rVert \\
			&= \lVert\left(Pa_{ij}P \otimes \mathbf{1}_{\mathcal{B}(H)}\right)_{i,j}\rVert \\
			&= \lVert\left(\tau(a_{ij})\right)_{i,j}\rVert
		\end{align}
		This shows in particular $\ker(\tau) = \ker((\tau \otimes \mathbf{I}^A)\Delta)$, whence $\alpha^\tau$ is well-defined.
		Moreover, we have proven that $\alpha^\tau$ is uci.
		The proof that $\beta^\tau$ is a well-defined uci map is analogous by exchanging the multiplicative unitary $W$ for $V$.
	\end{proof}
	
	For the rest of this section, we assume that the compact quantum $\mathbb{G}$ is coamenable with separable function algebra $A = \mathrm{C}(\mathbb{G})$.
	
	Let $L_A$ be a right/left/bi-invariant regular Lip-norm on the function algebra $A$, \emph{cf.\@} \autoref{rmk:Bi-invariant-regular-Lip-norms-exist}.
	Recall from \autoref{prop:Li-Thm-1.4} that the Lip-norm $L_A$ induces right/left/bi-invariant Lip-norms on $A^{(\Lambda)}$:
	
	\begin{corollary}\label{cor:Compression-is-a-CQMS}
		The operator system $A^{(\Lambda)}$ equipped with any of the induced Lip-norms $L_{A^{(\Lambda)}}^{\alpha^\tau}$, $L_{A^{(\Lambda)}}^{\beta^\tau}$ and $L_{A^{(\Lambda)}}^{\alpha^\tau,\beta^\tau} := \max\{L_{A^{(\Lambda)}}^{\alpha^\tau},L_{A^{(\Lambda)}}^{\beta^\tau}\}$ is a compact quantum metric space.
	\end{corollary}
	
	\begin{lemma}\label{lem:Compression-map-morphism}
		Let $L_A$ be a left invariant regular Lip-norm on the function algebra $A$ and let $L_{A^{(\Lambda)}}^{\alpha^\tau}$ be the induced Lip-norm on the operator system $A^{(\Lambda)}$.
		Then the compression map $\tau_\Lambda : A \rightarrow A^{(\Lambda)}$ is a morphism of compact quantum metric spaces.
		Analogous statements hold if $L_A$ is right or bi-invariant.
	\end{lemma}
	
	\begin{proof}
		We already noted in \autoref{rmk:Compression-map-ucp-onto} that the compression map $\tau$ is ucp.
		For Lip-norm contractivity, observe that the following holds, for all $a \in A$:
		\begin{align}
			L_{A^{(\Lambda)}}^{\alpha^\tau}(\tau(a))
			&= \sup_{\phi \in \mathcal{S}(A^{(\Lambda)})} L_A(\phi(\tau(a)_{(0)})a_{(1)}) \\
			&= \sup_{\phi \in \mathcal{S}(A^{(\Lambda)})} L_A(\tau^*\phi(a_{(0)})a_{(1)}) \\
			&= \sup_{\tau^*\phi \in \tau^*\mathcal{S}(A^{(\Lambda)})} L_A(\tau^*\phi(a_{(0)})a_{(1)}) \\
			&\leq L_A(a),
		\end{align}
		by left invariance of the Lip-norm $L_A$, where we used that $\tau^* : \mathcal{S}(A^{(\Lambda)}) \rightarrow \mathcal{S}(A)$ is an injection.
	\end{proof}
	
	\begin{definition}
		Let $L_A$ be a bi-invariant regular Lip-norm on the function algebra $A$ and let $L_{A^{(\Lambda)}}^{\alpha^\tau,\beta^\tau}$ be the induced bi-invariant Lip-norm on the operator system $A^{(\Lambda)}$ as in \autoref{cor:Compression-is-a-CQMS}.
		We call the compact quantum metric space $(A^{(\Lambda)},	L_{A^{(\Lambda)}}^{\alpha^\tau,\beta^\tau})$ the (bi-invariant) \emph{Peter--Weyl truncation} of the compact quantum group $\mathbb{G}$.
	\end{definition}
	
	In order to compare the Peter--Weyl truncations $(A^{(\Lambda)},	L_{A^{(\Lambda)}}^{\alpha^\tau,\beta^\tau})$ with the original compact quantum metric space $(A,L_A)$ using the criterion in \autoref{prop:Criterion-qGH-comparison} we need morphisms $\Phi : A \rightarrow A^{(\Lambda)}$ and $\Psi : A^{(\Lambda)} \rightarrow A$ whose compositions approximate the respective identity maps on $A$ and $A^{(\Lambda)}$ in Lip-norm.
	We take the compression map $\tau : A \rightarrow A^{(\Lambda)}$ as the morphism $\Phi$, so that it remains to find an appropriate candidate for the map $\Psi$.
	In earlier works on compact quantum metric spaces \cite{Rie04}, see also \cite{Sai09}, \cite{KK22}, \cite{vS21}, \cite{LvS24} and many more, these maps were inspired by Berezin quantization \cite{Ber75}, see also e.g.\@ \cite{Lan99}.
	For our purposes, rather than working with the adjoint of the compression map $\tau$ for a certain choice of inner products on $A$ and $A^{(\Lambda)}$, we follow the approach taken in \cite{GEvS23} to give a whole family of candidates for maps $\Psi : A^{(\Lambda)} \rightarrow A$, which we then show to have a member that satisfies the assumptions of \autoref{prop:Criterion-qGH-comparison}.
	We keep the subset $\Lambda \subseteq \widehat{\mathbb{G}}$ fixed.
	
	\begin{definition}\label{def:Symbol-map}
		Let $\phi \in \mathcal{S}(A^{(\Lambda)})$ be any state.
		We denote the associated slice map by $\sigma_\Lambda^\phi : A^{(\Lambda)} \rightarrow A$, i.e.\@
		\begin{align}
			\sigma_\Lambda^\phi(x) := \phi(x_{(0)})x_{(1)} 
			= (\phi \otimes \mathbf{I}^A) \alpha^{\tau}(x),
		\end{align}
		for all $x \in A^{(\Lambda)}$.
		We call the map $\sigma_\Lambda^\phi$ a \emph{symbol map}.
	\end{definition}
	
	\begin{notation}
		As we did for the compression map $\tau$, we will drop the subindex $\Lambda$ of the symbol map $\sigma_\Lambda^\phi$, whenever convenient.
	\end{notation}
	
	\begin{lemma}\label{lem:Symbol-map-morphism}
		Let $L_A$ be a regular Lip-norm on the function algebra $A$ and let $L_{A^{(\Lambda)}}^{\alpha^\tau}$ be the induced Lip-norm on the operator system $A^{(\Lambda)}$.
		Then, for every state $\phi \in \mathcal{S}(A^{(\Lambda)})$, the symbol map $\sigma^\phi : A^{(\Lambda)} \rightarrow A$ is a morphism of compact quantum metric spaces.
		
		Analogous statements hold if the operator system $A^{(\Lambda)}$ is equipped with one of the induced Lip-norms $L_{A^{(\Lambda)}}^{\beta^\tau}$ or $L_{A^{(\Lambda)}}^{\alpha^\tau,\beta^\tau}$.
	\end{lemma}
	
	\begin{proof}
		The symbol map $\sigma^\phi$, being the composition of the uci map $\alpha^\tau$ and the ucp map $\phi \otimes \mathbf{I}^A$, is ucp.
		Lip-norm contractivity of $\sigma^\phi$ follows from the definition of the induced Lip-norm:
		\begin{align}
			L_A(\sigma^\phi(x)) 
			= L_A(\phi(x_{(0)})x_{(1)}) 
			\leq \sup_{\psi \in \mathcal{S}(A^{(\Lambda)})} L_A(\psi(x_{(0)})x_{(1)}) 
			= L_{A^{(\Lambda)}}^{\alpha^\tau}(x)
		\end{align}
	\end{proof}
	
	Before we can apply \autoref{prop:Criterion-qGH-comparison}, we compute the compositions of the compression and symbol maps: 
	\begin{align}\label{eqn:Down-Up}
		\sigma^\phi\tau(a) = \phi(\tau(a)_{(0)}) \tau(a)_{(1)}
		= \tau^*\phi(a_{(0)}) a_{(1)},
	\end{align}
	for all $a \in A$, and
	\begin{align}\label{eqn:Up-Down}
		\tau\sigma^\phi(x)
		= \phi(x_{(0)}) \tau(x_{(1)})
		= \tau^*\phi(a_{(0)}) \tau(a_{(1)}),
	\end{align}
	for all $x \in A^{(\Lambda)}$ and $a \in A$ with $\tau(a) = x$, where we used (\ref{eqn:Induced-action-Sweedler}) in the last step.
	
	Recall that we are assuming that the compact quantum group $\mathbb{G}$ is coamenable with separable function algebra $A$.
	
	\begin{proposition}\label{prop:C1-Approximate-Order-Iso}
		Let $\phi \in \mathcal{S}(A^{(\Lambda)})$ be a state.
		Assume that $L_A$ is a regular Lipschitz seminorm on the function algebra $A$ with $\ker(L_A) = \mathbb{C}\mathbf{1}_A$.
		Let $L_A'$ be the induced right invariant regular Lipschitz seminorm on $A$ and let $L_{A^{(\Lambda)}}^{\beta^\tau}$ be the induced regular Lipschitz seminorm on $A^{(\Lambda)}$.
		Then the following inequalities hold:
		\begin{align}
			\lVert\sigma^\phi\tau(a)-a\rVert 
			&\leq 2 d^{L_A}(\tau^*\phi,\epsilon) L_A'(a),
		\end{align}
		and
		\begin{align}
			\lVert\tau\sigma^\phi(x)-x\rVert
			&\leq 2 d^{L_A}(\tau^*\phi,\epsilon) L_{A^{(\Lambda)}}^{\beta^\tau}(x),
		\end{align}
		for all elements $a \in A$ and $x \in A^{(\Lambda)}$, where we recall that $\epsilon \in \mathcal{S}(A)$ is the counit of the compact quantum group $\mathbb{G}$.
	\end{proposition}
	
	\begin{proof}
		The first inequality follows immediately from (\ref{eqn:Down-Up}) and  \autoref{cor:Lip-estimate-for-action-Alg}.
		Indeed, we have
		\begin{align}
			\lVert\sigma^\phi\tau(a)-a\rVert
			= \lVert\tau^*\phi(a_{(0)})a_{(1)} - \epsilon(a_{(0)})a_{(1)}\rVert
			\leq 2 d^{L_A}(\tau^*\phi,\epsilon) L_A'(a).
		\end{align}
		
		As for the second inequality, observe that, for all $a \in A$ with $\tau(a) = x$, we obtain the following, using (\ref{eqn:Up-Down}), the Kadison function representation and the Fubini theorem for slice maps:
		\begin{align}
			\lVert\tau\sigma^\phi(x) - x\rVert
			&= \lVert\tau^*\phi(a_{(0)})\tau(a_{(1)}) - \tau(a)\rVert \\
			&\leq 2 \sup_{\psi \in \mathcal{S}(A^{(\Lambda)})} \left|\psi \left(\tau^*\phi(a_{(0)})\tau(a_{(1)}) - \epsilon(a_{(0)})\tau(a_{(1)})\right)\right| \\
			&= 2 \sup_{\psi \in \mathcal{S}(A^{(\Lambda)})} \left|\tau^*\psi\left(\tau^*\phi(a_{(0)})a_{(1)} - \epsilon(a_{(0)})a_{(1)}\right)\right| \\
			&= 2 \sup_{\psi \in \mathcal{S}(A^{(\Lambda)})} | (\tau^*\phi - \epsilon) (a_{(0)} \tau^*\psi(a_{(1)})) | \\
			&\leq 2 \sup_{\psi \in \mathcal{S}(A^{(\Lambda)})} d^{L_A}(\tau^*\phi,\epsilon) L_A(a_{(0)} \tau^*\psi(a_{(1)})) \\
			&= 2 \sup_{\psi \in \mathcal{S}(A^{(\Lambda)})} d^{L_A}(\tau^*\phi,\epsilon) L_A(x_{(-1)} \psi(x_{(0)})) \\
			&= 2 d^{L_A}(\tau^*\phi,\epsilon) L_{A^{(\Lambda)}}^{\beta^\tau}(x).
		\end{align}
		Note that in the penultimate line we used that $a_{(0)} \otimes a_{(1)} = a_{(-1)} \otimes a_{(0)}$ and that $a_{(-1)} \tau^*\psi(a_{(0)}) = x_{(-1)} \psi(x_{(0)})$.
	\end{proof}
	
	\begin{corollary}\label{cor:Lip-norm-estimates-Up-Down-Down-Up}
		Assume that $L_A$ is a right invariant regular Lip-norm on the function algebra $A$ and let $L_{A^{(\Lambda)}}^{\beta^\tau}$ be the induced regular Lip-norm on $A^{(\Lambda)}$.
		Then for every positive real number $\varepsilon > 0$, there is a finite subset $\Lambda \subseteq \widehat{\mathbb{G}}$ and a state $\phi \in \mathcal{S}(A^{(\Lambda)})$ such that the following inequalities hold:
		\begin{align}
			\lVert\sigma_\Lambda^\phi\tau_\Lambda(a)-a\rVert 
			&\leq \varepsilon L_A(a), \text{ and } \\
			\lVert\tau_\Lambda\sigma_\Lambda^\phi(x)-x\rVert
			&\leq \varepsilon L_{A^{(\Lambda)}}^{\beta^\tau}(x),
		\end{align}
		for all $a \in A$ and $x \in A^{(\Lambda)}$.
	\end{corollary}
	
	\begin{proof}
		Using \autoref{prop:C1-Approximate-Order-Iso} together with the fact that $L_A = L_A'$ by right invariance, the claim follows from the assumption that $d^{L_A}$ metrizes the weak$^*$ topology on $\mathcal{S}(A)$, together with the weak$^*$ density of the subset of liftable states in $\mathcal{S}(A)$ as in \autoref{lem:Weak*-Density-Of-Liftable-States}. 
	\end{proof}
	
	\begin{theorem}\label{thm:Convergence-PW-truncations}
		Let $\mathbb{G}$ be a coamenable compact quantum group with separable function algebra $A = \mathrm{C}(\mathbb{G})$.
		Let $\mathcal{L} \subseteq \widehat{\mathbb{G}}$ be a net such that
		the induced net of projections $P_\Lambda$ onto $H_\Lambda := \bigoplus_{\gamma \in \Lambda} H_\gamma \otimes \overline{H_\gamma}$, for $\Lambda \in \mathcal{L}$, is a join semilattice and converges strongly to the identity on the Hilbert space $H := \mathrm{L}^2(\mathbb{G})$.
		Assume that $L_A$ is a bi-invariant regular Lip-norm on $A$ and denote by $L_{A^{(\Lambda)}}^{\alpha^\tau,\beta^\tau}$ the induced bi-invariant Lip-norm on the operator system $A^{(\Lambda)}$.
		Then the net of Peter--Weyl truncations $(A^{(\Lambda)},L_{A^{(\Lambda)}}^{\alpha^\tau,\beta^\tau})_{\Lambda \in \mathcal{L}}$ converges in operator Gromov--Hausdorff distance, i.e.
		\begin{align}
			\lim_{\Lambda \in \mathcal{L}} \mathrm{dist}^\mathrm{op}\left((A^{(\Lambda)},L_{A^{(\Lambda)}}^{\alpha^\tau,\beta^\tau}),(A,L_A)\right) = 0.
		\end{align}
	\end{theorem}
	
	\begin{proof}
		By \autoref{lem:Compression-map-morphism} and \autoref{lem:Symbol-map-morphism}, the compression map $\tau : A \rightarrow A^{(\Lambda)}$ and the symbol maps $\sigma^\phi : A^{(\Lambda)} \rightarrow A$ are morphisms of compact quantum metric spaces.
		By \autoref{cor:Lip-norm-estimates-Up-Down-Down-Up}, their compositions approximate the respective identity maps on $A$ and $A^{(\Lambda)}$ in Lip-norm.
		Thus, by \autoref{prop:Criterion-qGH-comparison}, we obtain convergence $(A^{(\Lambda)},L_{A^{(\Lambda)}}^{\alpha^\tau,\beta^\tau}) \rightarrow (A,L_A)$ in complete Gromov--Hausdorff distance which, by \autoref{prop:Equality-complete-operator-GH-distance} is equivalent to convergence in operator Gromov--Hausdorff distance.
	\end{proof}
	
	By \autoref{cor:Complete-GH-distance-bounds-all-others}, the same convergence result holds also in the distances $\mathrm{dist}_\mathrm{GH}$, $\mathrm{dist}_\mathrm{GH}^\mathrm{q}$, $\mathrm{dist}_n^\mathrm{s}$, for all $n \in \mathbb{N}$, and $\mathrm{dist}^\mathrm{s}$.

	\subsection{The case of a compact group}
	
	We compare our setup with the compact group case as in \cite{GEvS23}.
	To this end, let $G$ be a second countable compact group with a bi-invariant metric $d$, i.e.\@ $d(gh,gp) = d(hg,pg) = d(h,p)$, for all elements $g,h,p \in G$.
	Recall that we denote the comultiplication $\Delta : \mathrm{C}(G) \rightarrow \mathrm{C}(G \times G)$, $\Delta(f)(g,h) := f(gh)$, by $\alpha$ or $\beta$ whenever considered as a right or left coaction respectively.
	Recall furthermore that the Lipschitz constant $\mathrm{Lip}_d$ is a Lip-norm on $\mathrm{C}(G)$ and observe that it is bi-invariant in the sense of Li.
	Indeed, for any function $f \in \mathrm{C}(G)$ and state $\mu \in \mathcal{S}(\mathrm{C}(G))$ (i.e.\@ $\mu$ is a probability measure on $G$), we have
	\begin{align}
		\mathrm{Lip}_d((\mathbf{I}^{\mathrm{C}(G)} \otimes \mu)\alpha(f))
		&= \mathrm{Lip}_d \left(g \mapsto \int_G f(gh) \mathrm{d}\mu(h) \right) \\
		&= \sup_{g,p \in G} \frac{|\int_G f(gh) - f(ph) \mathrm{d}\mu(h)|}{d(g,p)} \\
		&\leq \int_G \sup_{g,p \in G} \frac{|f(gh)-f(ph)|}{d(g,p)} \mathrm{d}\mu(h) \\
		&\leq \mathrm{Lip}_d(f),
	\end{align}
	by right invariance of the metric $d$ and the fact that $\mu$ is a probability measure on $G$.
	Similarly, one shows that $\mathrm{Lip}_d((\mu \otimes \mathbf{I}^{\mathrm{C}(G)})\beta(f)) \leq \mathrm{Lip}_d(f)$.
	
	Denote by $U$ and $V$ the respective left and right regular representation of $G$ on the Hilbert space $\mathrm{L}^2(G)$, given by $U_g\xi(h) := \xi(g^{-1}h)$ and $V_g\xi(h) := \xi(hg)$ respectively, for all elements $g,h \in G$ and $\xi \in \mathrm{L}^2(G)$. 
	We write $\lambda, \rho$ for the strong$^*$-continuous left and right $G$-actions on $\mathcal{B}(\mathrm{L}^2(G))$ by conjugation with the left and right regular representation respectively, i.e.\@ $\lambda_g(T) := U_gTU_g^*$ and $\rho_g(T) := V_gTV_g^*$, for all elements $g \in G$ and operators $T \in \mathcal{B}(\mathrm{L}^2(G))$.
	It is straightforward to check that $\lambda_g(f)(h) = f(g^{-1}h)$ and $\rho_g(f)(h) = f(hg)$, for all elements $g,h \in G$ and functions $f \in \mathrm{C}(G)$ viewed as operators on the Hilbert space $\mathrm{L}^2(G)$ by pointwise multiplication.
	For all $T \in \mathcal{B}(\mathrm{L}^2(G))$, the authors of \cite{GEvS23} set
	\begin{align}\label{eqn:GEvS-Lip-norm}
		\lVert T\rVert_\lambda := \sup_{g \in G\setminus\{e\}} \frac{\lVert \lambda_g(T)-T\rVert}{d(g,e)}, \quad
		\lVert T\rVert_\rho := \sup_{g \in G\setminus\{e\}} \frac{\lVert\rho_g(T)-T\rVert}{d(g,e)},
	\end{align}
	and
	\begin{align}
		\lVert T \rVert_{\lambda,\rho} := \max\{\lVert T\rVert_\lambda, \lVert T\rVert_\rho\}.
	\end{align}
	It is straightforward to check that the Lipschitz constant of a function $f \in \mathrm{C}(G)$ coincides with the Lipschitz constants of the $\mathrm{C}(G)$-valued functions $g \mapsto \lambda_g(f)$ and $g \mapsto \rho_g(f)$, i.e.\@
	\begin{align}
		\mathrm{Lip}_d(f) = \lVert f \rVert_\lambda = \lVert f \rVert_\rho = \lVert f \rVert_{\lambda,\rho},
	\end{align}
	for all $f \in \mathrm{C}(G)$.
	
	Let $\Lambda \subseteq \widehat{G}$ be a set of equivalence classes of finite dimensional irreducible unitary representations of $G$ and let $P : \mathrm{L}^2(G) \rightarrow \bigoplus_{\gamma \in \Lambda} H_\gamma \otimes \overline{H_\gamma}$ be the associated orthogonal projection.
	The actions $\lambda$, $\rho$ commute with the compression map $\tau : \mathcal{B}(\mathrm{L}^2(G)) \ni T \mapsto PTP \in \mathcal{B}(P\mathrm{L}^2(G))$, so that we obtain $G$-actions on the operator system $P\mathrm{C}(G)P$ which we still denote by $\lambda$ and $\rho$ respectively.
	Note that, respectively being the composition of a norm-continuous $G$-action on $\mathrm{C}(G)$ and the compression map, these actions are norm-continuous.
	
	Denote by $\mathrm{Lip}^{\alpha^\tau,\beta^\tau}_{P\mathrm{C}(G)P}$ the bi-invariant Lip-norm on $P\mathrm{C}(G)P$ induced from the Lipschitz constant $\mathrm{Lip}_d$ by the coactions $\alpha^\tau : P\mathrm{C}(G)P \rightarrow P\mathrm{C}(G)P \otimes \mathrm{C}(G)$ and $\beta^\tau : P\mathrm{C}(G)P \rightarrow \mathrm{C}(G) \otimes P\mathrm{C}(G)P$ from \autoref{lem:Induced-action} in the sense of Li.
	I.e.
	\begin{align}
		\mathrm{Lip}^{\alpha^\tau,\beta^\tau}_{P\mathrm{C}(G)P}(x)
		:= \max\left\{\sup_{\phi \in \mathcal{S}(P\mathrm{C}(G)P)} \mathrm{Lip}_d (\phi(\alpha^\tau_\bullet(x))), \sup_{\phi \in \mathcal{S}(P\mathrm{C}(G)P)} \mathrm{Lip}_d(\phi(\beta^\tau_\bullet(x)))\right\},
	\end{align}
	for all $x \in P\mathrm{C}(G)P$.
	This Lip-norm is equivalent to the seminorm $\lVert \cdot \rVert_{\lambda,\rho}$ on $P\mathrm{C}(G)P$:
	
	\begin{lemma}\label{lem:Equivalence-induced-operator-seminorm}
		For all $x \in P\mathrm{C}(G)P$, the following holds:
		\begin{align}\label{eqn:equivalence-induced-and-left-seminorm}
			\frac{1}{2}\lVert x\rVert_{\lambda,\rho}
			\leq \mathrm{Lip}^{\alpha^\tau,\beta^\tau}_{P\mathrm{C}(G)P}(x)
			\leq \lVert x\rVert_{\lambda,\rho}.
		\end{align}
	\end{lemma}
	
	\begin{proof}
		Let $x \in P\mathrm{C}(G)P$.
		We show that
		\begin{align}\label{eqn:Lip-rho-alpha}
			\frac{1}{2} \lVert x \rVert_\rho 
			\leq \mathrm{Lip}^{\alpha^\tau}_{P\mathrm{C}(G)P}(x)
			\leq \lVert x \rVert_\rho,
		\end{align}
		where $\mathrm{Lip}^{\alpha^\tau}_{P\mathrm{C}(G)P}(x) := \sup_{\phi \in \mathcal{S}(P\mathrm{C}(G)P)} \mathrm{Lip}_d (\phi(\alpha^\tau_\bullet(x)))$.
		To this end, using the fact that $\mathrm{Lip}_d(f) = \lVert f \rVert_\lambda$, for all $f \in \mathrm{C}(G)$, we have:
		\begin{align}
			\mathrm{Lip}^{\alpha^\tau}_{P\mathrm{C}(G)P}(x)
			&= \sup_{\phi \in \mathcal{S}(P\mathrm{C}(G)P)} \lVert \phi(\alpha^\tau_\bullet(x)) \rVert_\lambda \\
			&= \sup_{\phi \in \mathcal{S}(P\mathrm{C}(G)P)} \sup_{g \in G \setminus \{e\}} \frac{\lVert \lambda_g(\phi(\alpha^\tau_\bullet(x))) - \phi(\alpha^\tau_\bullet(x)) \rVert_\infty}{d(g,e)} \\
			&= \sup_{\phi \in \mathcal{S}(P\mathrm{C}(G)P)} \sup_{g \in G \setminus \{e\}} \sup_{h \in G} \frac{\lvert \phi(\alpha^\tau_{g^{-1}h}(x)) - \phi(\alpha^\tau_h(x)) \rvert}{d(g,e)}
		\end{align}
		By the Kadison function representation, this last quantity is bounded below and above by respectively $\frac{1}{2}$ and $1$ times 
		\begin{align}
			\sup_{g \in G \setminus \{e\}} \sup_{h \in G} \frac{\lVert \alpha^\tau_{g^{-1}h}(x) - \alpha^\tau_h(x)\rVert}{d(g,e)}
			= \lVert \alpha^\tau_\bullet(x) \rVert_\lambda.
		\end{align}
		Note that, for all $g, h \in G$, $x \in P\mathrm{C}(G)P$ and $f \in \mathrm{C}(G)$ with $\tau(f) = x$, we have
		\begin{align}
			\alpha_{g^{-1}h}(x)
			&= ((\tau \otimes \mathbf{I}^{\mathrm{C}(G)})\Delta(f))(g^{-1}h) \\
			&= \tau(p \mapsto f(pg^{-1}h)) \\
			&= \tau(p \mapsto \rho_{g^{-1}h}(f)(p)) \\
			&= \rho_{g^{-1}h}(x).
		\end{align}
		This implies that 
		\begin{align}
			\lVert \alpha^\tau_\bullet(x) \rVert_\lambda
			= \sup_{g \in G \setminus \{e\}} \sup_{h \in G} \frac{\lVert \rho_{g^{-1}h}(x) - \rho_h(x)\rVert}{d(g,e)}
			= \|x\|_\rho,
		\end{align}
		by invariance of the metric $d$.
		Altogether we obtain (\ref{eqn:Lip-rho-alpha}).
		
		Similarly, we can show 
		\begin{align}
			\frac{1}{2} \lVert x \rVert_\lambda 
			\leq \mathrm{Lip}^{\beta^\tau}_{P\mathrm{C}(G)P}(x)
			\leq \lVert x \rVert_\lambda.
		\end{align}
		Together with (\ref{eqn:Lip-rho-alpha}) this yields the claim.
	\end{proof}
	
	We now obtain \cite[Theorem 16]{GEvS23} as a corollary of our \autoref{thm:Convergence-PW-truncations}:
	
	\begin{corollary}
		Let $G$ be a compact group and let $\mathcal{L} \subseteq \widehat{G}$ be a net of finite dimensional irreducible unitary representations such that the induced net of projections $P_\Lambda$ onto $H_\Lambda := \bigoplus_{\gamma \in \Lambda} H_\gamma \otimes \overline{H_\gamma}$, for $\Lambda \in \mathcal{L}$, is a join semilattice and converges strongly to the identity on the Hilbert space $H := \mathrm{L}^2(G)$.
		Assume that the group $G$ is equipped with a bi-invariant metric $d$.
		Then the net of compact metric spaces $(\mathcal{S}(P_\Lambda \mathrm{C}(G) P_\Lambda),d^{\lVert\cdot\rVert_{\lambda,\rho}})$ converges to the compact metric space $(\mathcal{S}(\mathrm{C}(G)),d^{\mathrm{Lip}})$ in Gromov--Hausdorff distance.
	\end{corollary}
	
	\begin{proof}
		By \autoref{lem:Equivalence-induced-operator-seminorm} and \autoref{thm:Convergence-PW-truncations}, we have convergence of the net of compact quantum metric spaces $(P_\Lambda \mathrm{C}(G) P_\Lambda,\lVert\cdot\rVert_{\lambda,\rho})$ to the compact quantum metric space $(\mathrm{C}(G),\mathrm{Lip})$ in operator Gromov--Hausdorff distance,	which implies the claim, by \autoref{cor:Complete-GH-distance-bounds-all-others}.
	\end{proof}
	
	\begin{remark}
		Note that the two seminorms $L_{P\mathrm{C}(G)P}^{\alpha^\tau,\rho^\tau}$ and $\lVert\cdot\rVert_{\lambda,\rho}$ coincide on the self-adjoint subspace $(P\mathrm{C}(G)P)_\mathrm{sa}$.
		Indeed, this follows from the equality $\lVert x\rVert = \sup_{\phi \in \mathcal{S}(X)}|\phi(x)|$, for all self-adjoint elements of an operator system $X$, by the Kadison function representation.
	\end{remark}

\end{document}